\documentclass[12pt]{amsart}
\usepackage{latexsym, amscd, a4wide, verbatim}
\usepackage{amsmath,amssymb,amsthm,amsfonts,enumerate,array,color,lscape,fancyhdr,layout,pst-all}

\usepackage[english]{babel}
\usepackage[latin1]{inputenc}
\usepackage{young}

\usepackage[hcentermath]{youngtab}

\usepackage{graphicx}
\usepackage{float}
\usepackage{pstricks, mathtools}

\newcounter{numberofremark}
\newcommand\nothing[1]{}
\usepackage[active]{srcltx}
\usepackage[hyperindex,plainpages,bookmarksnumbered]{hyperref}
\usepackage[all]{xy}

\newcommand{\dcl}{\DeclareMathOperator}
\dcl\cdet{cdet} \dcl\Sp{Specm} \dcl\depth{depth} \dcl\im{Im} \dcl\Span{Span} \dcl\id{id} \dcl\Ker{Ker} \dcl\Specm{Specm}
\dcl\Supp{Supp} \dcl\codim{codim} \dcl\Y{Y} \dcl\gl{\mathfrak{gl}}    \dcl\U{U} \dcl\T{T} \dcl\soc{soc} 
\dcl\qdet{qdet} \dcl\sgn{sgn} \dcl\gr{gr} \dcl\diag{diag}
\dcl\g{\mathfrak{g}} \dcl\C{\mathbb C} \dcl\dd{{\mathrm d}}
 \dcl\Tab{Tab}

\newcommand{\mf}[1]{\mathfrak{#1}}
\newcommand{\Z}{\ensuremath{\mathbb{Z}}}
\newcommand{\h}{\ensuremath{\mathfrak{h}}}

\newlength\yStones
\newlength\xStones
\newlength\xxStones

\makeatletter
\def\Stones{\pst@object{Stones}}
\def\Stones@i#1{%
  \pst@killglue%
  \begingroup%
  \use@par%
  \setlength\xxStones{\xStones}%
  \expandafter\Stones@ii#1,,\@nil
  \endgroup
  \global\addtolength\xStones{0.6cm}%
  \global\addtolength\yStones{-7.5mm}}%
\def\Stones@ii#1,#2,#3\@nil{%
  \rput(\xxStones,\yStones){%
    \psframebox[framesep=0]{%
      \parbox[c][6mm][c]{11mm}{\makebox[11mm]{$#1$}}}}%
  \addtolength\xxStones{1.2cm}%
  \ifx\relax#2\relax\else\Stones@ii#2,#3\@nil\fi}
\makeatother

\def\Stone#1{\fbox{\makebox[11mm]{\strut#1}}\kern2pt}

\newtheorem{theorem}{Theorem}[section]
\newtheorem{lemma}[theorem]{Lemma}
\newtheorem{corollary}[theorem]{Corollary}
\newtheorem{proposition}[theorem]{Proposition}
\newtheorem{example}[theorem]{Example}
\newtheorem{remark}[theorem]{Remark}

\newtheorem{definition}[theorem]{Definition}
\newtheorem*{lemma*}{Lemma}

\begin{document}

\title[Highest weight modules and non-standard Gelfand-Tsetlin subalgebras]{Highest weight modules with respect to non-standard Gelfand-Tsetlin subalgebras}

\author{Juan Camilo Arias}
\address{\noindent Universidad del Rosario, Bogot\'a, Colombia} \email{juancamil.arias@urosario.edu.co}

\author{Oscar Morales}
\address{\noindent Universidade Federal do Par\'a, Bel\'em--PA , Brasil} \email{oscar@ufpa.br}

\author{Luis Enrique Ramirez}
\address{\noindent
Universidade Federal do ABC,  Santo Andr\'e--SP, Brasil} \email{luis.enrique@ufabc.edu.br}



\maketitle

\begin{abstract}
In this paper, we study realizations of highest weight modules for the complex Lie algebra $\mathfrak{gl}_n$ with respect to non-standard Gelfand-Tsetlin subalgebras. We also provide sufficient conditions for such subalgebras to have a diagonalizable action on these realizations.
\end{abstract}

\makeatletter
\def\@roman#1{\romannumeral#1}
\makeatother

\section{Introduction}
By definition, a module of the complex Lie algebra $\mathfrak{gl}_n$ is a Gelfand-Tsetlin module if it has a generalized eigenspace decomposition over a certain maximal commutative subalgebra of the universal enveloping algebra of $\mathfrak{gl}_n$. Gelfand-Tsetlin modules generalize the classical realization of the simple finite-dimensional modules of $\mathfrak{gl}_n$ via the so-called Gelfand-Tsetlin tableaux introduced in \cite{GT50}.\\ 

The category of Gelfand-Tsetlin modules plays an important role in many areas of mathematics and theoretical physics, and have been studied extensively since its origins \cite{DFO94}. Recently, it has attracted considerable attention after the introduction of the theories of singular Gelfand-Tsetlin modules \cite{FGR16}, and relation Gelfand-Tsetlin modules \cite{FRZ19}.  On one hand, the construction of singular Gelfand-Tsetlin modules deal with a generalization of the classical Gelfand-Tsetlin formulas in order to construct modules with a tableau-type realization, and the explicit constructions includes a large class of modules for which the action of the Gelfand-Tsetlin subalgebra is not diagonalizable. Several important results in this direction were obtained in \cite{EMV20, FGR17a, FGR17b, FGR21, FGRZ20a, FGRZ20b, RZ18, Vis18, Web24}, among others. On the other hand, the construction of relation Gelfand-Tsetlin modules deals with conditions to avoid singularities, and unify several known constructions of Gelfand-Tsetlin modules, including finite dimensional modules \cite{GT50}, generic modules \cite{DFO94}, and some families of modules constructed by relaxing the conditions on the Gelfand-Tsetlin construction of finite dimensional modules (see for instance \cite{GG65}, \cite{LP79}, \cite{Maz98}, \cite{Maz03}). This class of modules has as main advantage the combinatorial nature of the construction and explicitness of the action, which is given by the classical Gelfand-Tsetlin formulas.\\

Despite the category of Gelfand-Tsetlin modules is attached to a maximal commutative subalgebra, early results in \cite{DFO91, DFO94} stated that the categories of modules over any two  Gelfand-Tsetlin subalgebras are equivalent. Because of this outcome, most of the known results and constructions were obtained for modules over the standard Gelfand-Tsetlin subalgebras. In this paper we study realizations of highest weight modules as twisted relation Gelfand-Tsetlin modules, in particular we provide structural results for the module with respect to non-standard Gelfand-Tsetlin subalgebra. As an application, we use this approach and localization functors to construct several classes of modules which are not necessarily highest weight modules and can be described by a twisted action of the classical Gelfand-Tsetlin formulas.\\

\noindent
Let us briefly summarize the content of this paper. In section \ref{S:Preliminaries}, we set basic notations on Lie algebras and root datum. In section \ref{S:Relation GT modules} we summarize the concepts and known results about relation Gelfand-Tsetlin modules and relation graphs. In section \ref{S:Action on relation graphs}, using an action of the symmetric group on a convenient set of graphs, we explicitly construct highest weight vectors for the realizations of a fixed highest weight module as a twisted relation module. Finally, section \ref{S:Main results and applications} contains the main results of the paper, we study a special class of weights, called relation weights, which provide well behaved Gelfand-Tsetlin modules with respect to non-standard Gelfand-Tsetlin subalgebras. We construct several families of relation weights, and use localization functors to construct some families of relation modules that are not highest weight modules.

\medskip
\noindent{\bf Acknowledgements.} O.M. was partially supported by CMUP, member of LASI, which is financed by national funds through FCT -- Funda\c c\~ao para a Ci\^encia e a Tecnologia, I.P., under the projects with references UIDB/00144/2020 and UIDP/00144/2020. L.E.R. was partially supported by CNPq grant 316163/2021-0, and Fapesp grant 2022/08948-2.

\section{Preliminaries}\label{S:Preliminaries}

 In this section, we fix some basic notations and definitions regarding Lie algebras and partitions of root systems.

\subsection{Root datum } Let us fix $n\geq 2$, and consider the reductive Lie algebra $\mf{gl}_n$ of $n\times n$ matrices over the complex numbers. We denote by $\mf{sl}_n$ the simple Lie subalgebra of traceless matrices in $\mf{gl}_n$. By $E_{ij}$ we denote the matrix in $\mf{gl}_n$ with $1$ in the $(i,j)$-th position and zero elsewhere. It is known that $\mf{gl}_n = \mf{sl}_n \oplus \mf{s}_n$, where $\mf{s}_n$ is the Lie algebra of scalar matrices of size $n\times n$, and that $\mf{sl}_n$ has a triangular decomposition $\mf{sl}_n \cong \mf{n}_{-}\oplus \h \oplus \mf{n}_{+}$. Here, $\mf{h}$ is the (standard) Cartan subalgebra of $\mf{sl}_n$ which is the abelian Lie algebra generated by the matrices $E_{i,i} - E_{i+1,i+1}$, where $i=1, \ldots, n-1$, and $\mf{n}_{+},\ \mf{n}_{-}$ are the positive and negative nilpotent Lie subalgebras of $\mf{sl}_n$ generated by $\{E_{ij}\}_{1\leq i< j\leq n}$ and $\{E_{ji}\}_{1\leq i< j\leq n}$ respectively. We denote by $\mf{b} = \h \oplus\mf{n}_{+}$ the standard Borel subalgebra of $\mf{sl}_n.$\\

Let $\Delta$ be the standard root system of $\mf{gl}_n$ (and $\mf{sl}_n$) with set of simple roots given by $\Pi=\{\alpha_i = \varepsilon_i - \varepsilon_{i+1}\ |\ i=1, \ldots, n-1\}$. For $1\leq i < j \leq n$ we write $\alpha_{ij} = \varepsilon_i - \varepsilon_{j}$, and $\alpha_{ij} := -\alpha_{ji}$ if $i>j$ (note that $\alpha_{i,i+1} = \alpha_i$).  $\Delta^{+} = \{\alpha_{ij}\ |\ 1\leq i < j \leq n \}$ denotes the set of positive roots, and $\Delta^{-} = -\Delta^{+}$ denotes the set of negative roots.  We denote by $W$ the Weyl group of $\mf{gl}_n$ and $\mf{sl}_n$. It is isomorphic to the symmetric group $S_n$ and it is generated by the simple reflections $s_k := s_{\alpha_k}$ that act on the set of roots by permuting $\varepsilon_k$ and $\varepsilon_{k+1}$. Note that for $\sigma\in S_n$, $\sigma(\alpha_{ij}) =\alpha_{\sigma(i) \sigma(j)}$, and we have that $\sigma(\alpha_{ij})$ is a positive root if $\sigma(i) < \sigma (j)$, and it is a negative root if $\sigma(i) > \sigma (j)$.\\

Denote by $P \subset \h^*$ the integral weight lattice of $\mf{sl}_n$, and let $Q \subset P$ be the integral root lattice. It is known that $P$ is a free abelian group of rank $n-1$ generated by the fundamental weights $\varpi_1, \ldots, \varpi_{n-1}$, determined by  the relations $\langle \varpi_i, \alpha_j^{\vee} \rangle = \delta_{ij}$, where $\alpha_j^{\vee}$ is the coroot associated with $\alpha_j$ and $\langle - , - \rangle$ is the Cartan-Killing form. Set $\rho$ to be the half sum of the positive roots, and let $P^{+} = \{ \lambda \in \h^*\ |\ \langle \lambda + \rho, \alpha_i^{\vee} \rangle \in \Z_{> 0},\ i=1, \ldots, n-1\}$ denote the set of dominant integral weights. Due to the natural identification of $\h^*$ with $\C^{n-1}$, we will identify $P$ with $\Z^{n-1}$ and $P^{+}$ with $\Z^{n-1}_{\geq 0}$ where $\lambda \in P$ corresponds with the vector $(\lambda_1, \ldots, \lambda_{n-1})$ for $\lambda_i = \langle \lambda, \varpi_i \rangle$. The usual action of $W$ on $\h^*$ is given by $s_i(\lambda) := \lambda - \langle \lambda , \alpha_i^{\vee} \rangle \alpha_i$ and the dot action $s_i\cdot \lambda := s_i(\lambda + \rho) - \rho = \lambda - \langle \lambda + \rho, \alpha_i^{\vee}\rangle\alpha_i$.

\subsection{Partitions and quase-partitions of \texorpdfstring{$\Delta$}{Delta} }

Let $S$ be a subset of $\Delta$. We say that $S$ is \emph{closed} if for any $\alpha, \beta \in S$ with $\alpha + \beta \in \Delta$ then $\alpha + \beta \in S$, $S$ is called a \emph{closed partition} if it is closed, $\Delta = S \cup (-S) $, and $S\cap (-S) = \emptyset$. The standard closed partition of $\Delta$ is given by the set of positive roots $\Delta^{+}$, and it is known that for any semisimple Lie algebra, all the closed partitions are conjugate to $\Delta^{+}$ by the usual action of the Weyl group. Moreover, for a partition to be closed, it is equivalent to require that it is conjugate to the set of positive roots.\\

A subset $\mathcal{Q} \subset \Delta$ is called a \textit{quase-partition} of $\Delta$ if  $|\Delta|=2 |\mathcal{Q}|$, and $\alpha \in \mathcal{Q}$ if and only if $-\alpha \notin \mathcal{Q}$. Clearly, any closed partition is a quase-partition of $\Delta$. On the other hand, a quase-partition is a closed partition if and only if it is closed, which, in turn, is equivalent to the existence of $\sigma \in W$ such that $\mathcal{Q}=\sigma(\Delta^+)$. Moreover, for any quase-partition $\mathcal{Q}$ of $\Delta$ we have  $\mathfrak{h}=\operatorname{span}_{\mathbb C} \{ h_\alpha \ | \ \alpha \in \mathcal{Q} \} $.

\section{Relation Gelfand-Tsetlin modules}\label{S:Relation GT modules}

In this section, we collect the known facts that we will need about Gelfand-Tsetlin modules with respect to the standard Gelfand-Tsetlin subalgebra.\\

Associated with any $\sigma\in S_n$ and $k\leq n$, denote by $B_k$ the subalgebra of $\mathfrak{gl}_n$ generated by $\{ E_{\sigma(i)\sigma(j)}\,|\, i,j=1,\ldots,k \}$. Note that $B_1\subset B_2\subset \ldots \subset B_n$, and $B_{k}\simeq \mathfrak{gl}_k$ for any $1\leq k\leq n$. Such a chain of subalgebras induces a chain  $U_1\subset U_2\subset\ldots\subset U_n$ for the universal enveloping algebras $U_{k}=U(B_k)$, $1\leq k\leq n$. Let $Z_{k}$ be the center of $U_{k}$, and denote by $\Gamma_{\sigma}$ the subalgebra of $U(\mf{gl}_n)$ generated by $\bigcup\limits_{k=1}^n
Z_k$. Any subalgebra of the form $\Gamma_{\sigma}$ for some $\sigma$ in $S_n$ is called \emph{Gelfand-Tsetlin subalgebra} of type $A$, we also refer to $\Gamma:=\Gamma_{id}$ as the \emph{standard Gelfand-Tsetlin subalgebra} of $U(\mf{gl}_n)$.

\begin{definition} Let $\widetilde{\Gamma}$ be any Gelfand-Tsetlin subalgebra of type $A$. A finitely generated $U(\mf{gl}_n)$-module
$M$ is called a $\widetilde{\Gamma}$-\emph{Gelfand-Tsetlin module} if 
\begin{equation}
M=\bigoplus_{\chi\in\widetilde{\Gamma}^{*}}M(\chi)
\end{equation}
where $
M(\chi)=\{v\in M\ |\ \forall g\in\widetilde{\Gamma}\text{, there exists } k\in\mathbb{Z}_{>0} \text{ such that } (g-\chi(g))^{k}v=0 \}.$
\end{definition}

\subsection{Relation modules}\label{S: Rel Modules} 
 In \cite{FRZ19} the class of relation Gelfand-Tsetlin modules was introduced as an attempt to unify several known constructions of Gelfand-Tsetlin modules with diagonalizable action of the standard Gelfand-Tsetlin subalgebra. This section is devoted to recall the construction and main properties of relation Gelfand-Tsetlin modules with respect to the standard Gelfand-Tsetlin subalgebra.\\

Denote by $\mathfrak{V}$ the set $\{(i,j)\ |\ 1\leq j\leq i\leq n\}$ arranged in a triangular configuration with $n$ rows, where the $k$-th row is written as $((k,1),\ldots,(k,k))$, and the top row is given by the $n$-th row. From now on, we only consider directed graphs $G$ with set of vertices $\mathfrak{V}$, such that the 
only possible arrows are those connecting vertices in consecutive rows or vertices in the $n$-th row. We will assume that $G$ does not contain oriented cycles or multiple arrows. 

\begin{definition}\label{Def: critical} A graph $G$ will be called \emph{non-critical} if for any two vertices in the same row $k<n$ and the same connected component (of the unoriented graph associated with $G$), there exists an oriented path in $G$ from one vertex to the other. \end{definition}

\begin{example} Consider the following two graphs for $n=4$: 

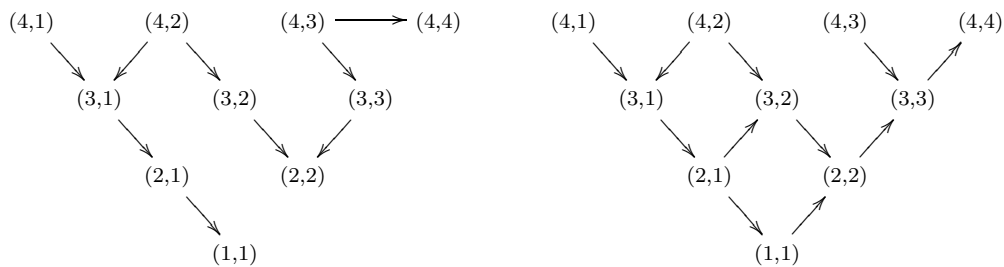
\begin{tabular}{c c}
\xymatrixrowsep{0.5cm}
\xymatrixcolsep{0.1cm}\xymatrix @C=0.1em { 
& \scriptstyle{(4,1)}\ar[rd]& &\scriptstyle{(4,2)}\ar[dl] \ar[rd]   & &\scriptstyle{(4,3)}\ar[rd] \ar[rr] &   &\scriptstyle{(4,4)} \\
 & & \scriptstyle{(3,1)}  \ar[rd]   & &\scriptstyle{(3,2)}\ar[rd]  & &\scriptstyle{(3,3)}\ar[dl]  \\
&    &     &\scriptstyle{(2,1)} \ar[rd] & &\scriptstyle{(2,2)}  & \\
&  &    & &\scriptstyle{(1,1)} & &\\
}\\

\end{tabular}   
\begin{tabular}{c c}
\xymatrixrowsep{0.5cm}
\xymatrixcolsep{0.1cm}\xymatrix @C=0.1em { 
& \scriptstyle{(4,1)}\ar[rd]& &\scriptstyle{(4,2)}\ar[dl] \ar[rd]   & &\scriptstyle{(4,3)}\ar[rd] &   &\scriptstyle{(4,4)} \\
 & & \scriptstyle{(3,1)}  \ar[rd]   & &\scriptstyle{(3,2)}\ar[rd] & &\scriptstyle{(3,3)}\ar[ru]  \\
&    &     &\scriptstyle{(2,1)}\ar[ru] \ar[rd] & &\scriptstyle{(2,2)}\ar[ru]  & \\
&  &    & &\scriptstyle{(1,1)}\ar[ru] & &\\
}\\

\end{tabular}   \\
The graph to the left is critical, and the one to the right is non-critical.
\end{example}

We write vectors in $\mathbb{C}^{\frac{n(n+1)}{2}}$ as ordered tuples $L=(l_{n1},\ldots,l_{nn},\ldots,l_{22},l_{21},l_{11})$ indexed by elements in $\mathfrak{V}$, and by $T(L)$ we denote the triangular configuration induced by the triangular configuration of $\mathfrak{V}$. Such arrays will be called \emph{Gelfand-Tsetlin tableaux}.

\begin{definition} Let $G$ be any graph and $T(L)$ any Gelfand-Tsetlin tableau.

\begin{itemize}
\item[(i)] We say that \emph{$T(L)$ satisfies $G$} if
\begin{itemize}
\item[(a)] $l_{ij}-l_{rs}\in \mathbb{Z}_{\geq 0}$ whenever $(i,j)$ and $(r,s)$ are connected by a horizontal arrow or an arrow pointing down.
\item[(b)] $l_{ij}-l_{rs}\in \mathbb{Z}_{> 0}$ whenever $(i,j)$ and $(r,s)$ are connected by an arrow pointing up.
\end{itemize}
\item[(ii)] We say that \emph{$T(L)$ is a $G$-realization} if
\begin{itemize}
\item[(a)] $T(L)$ satisfies $G$.
\item[(b)]  For any $1\leq k\leq n-1$, we have $l_{ki}-l_{kj}\in \mathbb{Z} $ only if $(k,i)$ and $(k,j)$ are in the same connected component of the unoriented graph associated with $G$. 
\end{itemize}

\item[(iii)] If $T(L)$ is a $G$-realization, by ${\mathcal B}_{G}(T(L))$ we denote the set of all $G$-realizations of the form $T(L+z)$, with $z\in {\mathbb Z}^\frac{n(n+1)}{2}$ such that $z_{ni}=0$ for  $1\leq i\leq n$. By $V_{G}(T(L))$ we denote the complex vector space spanned by ${\mathcal B}_{G}(T(L))$.
\end{itemize}
\end{definition}

\begin{definition}
A non-critical graph $G$ is called \emph{relation graph} if for any $G$-realization $T(L)$, the vector space $V_{G}(T(L))$ has a structure of $\mathfrak{gl}_n$-module, endowed with the action of $\mathfrak{gl}_n$ given by the Gelfand-Tsetlin formulas.
 \begin{equation}\label{GT-formulas}
\begin{split}
E_{k,k+1}(T(L))&=-\sum_{i=1}^{k}\left(\frac{\prod_{j=1}^{k+1}(l_{ki}-l_{k+1,j})}{\prod_{j\neq i}^{k}(l_{ki}-l_{kj})}\right)T(L+\delta^{ki}),\\ E_{k+1,k}(T(L))&=\sum_{i=1}^{k}\left(\frac{\prod_{j=1}^{k-1}(l_{ki}-l_{k-1,j})}{\prod_{j\neq i}^{k}(l_{ki}-l_{kj})}\right)T(L-\delta^{ki}),\\
E_{kk}(T(L))&=\left(k-1+\sum_{i=1}^{k}l_{ki}-\sum_{i=1}^{k-1}l_{k-1,i}\right)T(L),
\end{split}
\end{equation}
where $T(L\pm\delta^{ki})$ denotes the Gelfand-Tsetlin tableau obtained from $T(L)$ by adding $\pm 1$ to the $(k,i)$-th entry of $T(L)$. By definition, whenever the new tableau $T(L\pm\delta^{ki})$ is not a $G$-realization, the corresponding summand of $E_{k,k+1}(T(L))$ or $E_{k+1,k}(T(L))$ is zero. Modules isomorphic to $V_{G}(T(L))$ for some relation graph $G$ will be called \emph{relation modules}.
 \end{definition}

Recall that, for $X = (x_1, \ldots, x_m)\in \mathbb{C}^{m}$ and $\sigma\in S_m$, we denote by $\sigma(X)$ the vector $(x_{\sigma^{-1}(1)}, \ldots, x_{\sigma^{-1}(m)})$.

\begin{remark}\label{general formulas}
Using the action of the generators, it is possible to describe explicitly the action of $E_{\ell m}$ with $m\neq \ell$ for any relation module (see \cite[Proposition 3.13]{FGR16}). Indeed, let $\widetilde{S}_t$ denote the subset of $S_t$ consisting of the transpositions $(1,i)$, $i=1,...,t$ for $\ell < m $, set $\Phi_{\ell m} =  \widetilde{S}_{m-1} \times\cdots\times \widetilde{S}_{\ell}$, and for $\ell > m$ set $\Phi_{\ell m} = \Phi_{m \ell}$. Then
$$E_{\ell m} (T(L))= \sum_{\sigma \in \Phi_{\ell m}} e_{\ell m} (\sigma (L)) T(L+\sigma(\varepsilon_{\ell m})),$$
where $e_{\ell m}(L)$ is a rational function on the entries of the tableau $T(L)$, for $\ell<m$, $\varepsilon_{\ell m}:=
\delta^{\ell,1}+\delta^{\ell+1,1}+\ldots+\delta^{m-1,1}$, and $\varepsilon_{m\ell}=-\varepsilon_{\ell m}$, and whenever the tableau $T(L+\sigma(\varepsilon_{\ell m}))$ is not a $G$-realization, the corresponding summand of $E_{\ell m} (T(L))$ is zero by definition.

 \end{remark}

\begin{definition} Let $G$ be any graph.
\begin{itemize}
    \item[(i)] $G$ will be called \emph{ordered} if for any $1\leq k\leq n$ and $1\leq i<j\leq k$ there are no directed paths from $(k,j)$ to $(k,i)$. 
    \item[(ii)] For an ordered graph $G$, and $1\leq i<j\leq k\leq n-1$, we call $((k,i),(k,j))$ an \emph{adjoining pair in $G$} if there is a directed path from $(k,i)$ to $(k,j)$ and  whenever $ i<t< j$, there are no directed paths from $(k,i)$ to $(k,t)$ or from $(k,t)$ to $(k,j)$. 
    \item[(iii)] An ordered graph $G$ \emph{ has crosses} if it contains a subgraph of the form:
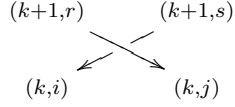
\begin{center}
\begin{tabular}{c c }
\xymatrixrowsep{0.5cm}
\xymatrixcolsep{0.1cm}
\xymatrix @C=0.2em{
 \scriptstyle{(k+1,r)}\ar[rrrd] &   & &  \scriptstyle{(k+1,s)}\ar[llld] |!{[d];[lll]}\hole    & \\
 \scriptstyle{(k,i)}& &    &\scriptstyle{(k,j)}   & }
\end{tabular}
\end{center}
 with $1\leq r<s\leq k+1$, and $1\leq i< j\leq k$.

 \item[(iv)] An ordered graph $G$ satisfies the $\Diamond$-condition if for every adjoining pair of vertices $((k,i),(k,j))$, one of the following is a subgraph of $G$: 
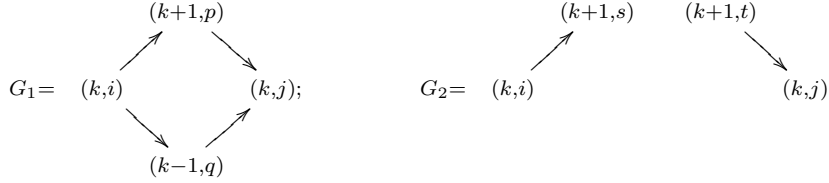
\begin{center}
\begin{tabular}{c c c c}
\xymatrixrowsep{0.5cm}
\xymatrixcolsep{0.1cm}
\xymatrix @C=0.2em{
  &   &\scriptstyle{(k+1,p)}\ar[rd]   &   & \\
 \scriptstyle{G_1}=  &\scriptstyle{(k,i)}\ar[rd] \ar[ru]  &    &\scriptstyle{(k,j)};   &  \\
   &   &\scriptstyle{(k-1,q)}\ar[ru]   &   & }
&\ \ &
\xymatrixrowsep{0.5cm}
\xymatrixcolsep{0.1cm}\xymatrix @C=0.2em {
   &   &\scriptstyle{(k+1,s)}    &   &\scriptstyle{(k+1,t)}\ar[rd]&& \\
  \scriptstyle{G_{2}}= &\scriptstyle{(k,i)} \ar[ru]  & &   & & \scriptstyle{(k,j)} \\
   &   &   &   & &&}
\end{tabular}
\end{center}

for some $1\leq q\leq k-1$, $1\leq p\leq k+1$, or $1\leq s<t\leq k+1.$
    
\end{itemize}
\end{definition}

We have the following characterization of relation graphs.

\begin{theorem}\label{relation mod} An ordered, non-critical, cross-less graph $G$ is a relation graph if and only if every connected component of $G$ satisfies the $\Diamond$-condition.
\end{theorem}

\begin{proof}
See \cite[Theorem 4.33]{FRZ19}.
\end{proof}

 For any directed graph $G$ with finite vertices and without cycles, there is a unique graph with the same set of vertices and the same reachability relations as the original graph, with a minimal set of arrows. Such graph will be called the transitive reduction of $G$ and will be denoted by $\bar{G}$. We call $G$ \emph{transitive reduced} if  $G=\bar{G}$.\\

\begin{definition}\label{Graph associated with a tableau}
Associated with any tableau $T(L)$, by $G(T(L))$ we denote the transitive reduction of the graph with set of vertices  $\mathfrak{V}$ and an arrow from $(i,j)$ to $(r,s)$ if 
\begin{itemize}
\item[(i)]  $i=r+1$, and $l_{ij}-l_{rs}\in\mathbb{Z}_{\geq 0}$, or 
\item[(ii)] $i=r-1$, and $l_{ij}-l_{rs}\in\mathbb{Z}_{>0}$, or
\item[(iii)] $i=r=n$, $j\neq s$, and $l_{ij}-l_{rs}\in\mathbb{Z}_{\geq 0}$.
\end{itemize}
\end{definition}
\begin{remark} Definition \ref{Graph associated with a tableau} coincides with the concept of maximal set of relations associated with a tableau introduced in \cite{FRZ19}, in particular, if $G$ is a relation graph and $T(L)$ a $G$-realization, then $V_{G}(T(L))$ is simple if and only if $G=G(T(L))$ (see \cite[Theorem 5.6]{FRZ19}). We should also note that graphs of the form $G(T(L))$ are not necessarily ordered graphs. However, there is an isomorphism between the modules obtained by any tableau in the orbit of $T(L)$ under the natural action of the group of $S_n\times S_{n-1}\times\cdots\times S_1$ on Gelfand-Tsetlin tableaux (see \cite[$\S$ 4.3]{FRZ19}). In particular, in order to check the $\Diamond$-condition for $G(T(L))$, it is enough to verify the condition for $G(T(\tilde{L}))$, where $T(\tilde{L})$ is any tableau in the orbit of $T(L)$ satisfying $\tilde{\ell}_{ki}-\tilde{\ell}_{kj}\in\mathbb{Z}_{\geq 0}$ implies $i\leq j$, whenever $k\leq n$.
 \end{remark}
 
One of the main reasons for our interest in relation modules is the well-behave action of the Gelfand-Tsetlin subalgebra. Indeed,

\begin{theorem}\label{mult_rel_mods} For any relation graph $G$, the module $V_G(T(L))$ is a Gelfand-Tsetlin module with diagonalizable action of the generators of the standard Gelfand-Tsetlin subalgebra. Moreover, for every $\chi\in\Gamma^{*}$ in the support of $V_G(T(L))$, $\dim(V_G(T(L))(\chi))=1$.
\end{theorem}

\begin{proof} See \cite[Theorem 5.3, and Theorem 5.8]{FRZ19}. 
\end{proof}

\section{Action on relation graphs}\label{S:Action on relation graphs}
 In this section, we describe the relation graphs leading to highest weight modules with respect to non-standard Gelfand-Tsetlin subalgebras, and describe a method to construct highest weight vectors for such modules.

\subsection{An $S_n$-action on relation graphs} 

From now on, $\h$ will denote the standard Cartan subalgebra of $\mathfrak{gl}_n$. By $G_{\h}$ we denote the directed graph with the set of vertices $\mathfrak{V}$, and arrows from $(i,j)$ to $(i-1,j)$ for any $ 1\leq j< i\leq n$. For example, in the case of $n=5$, the graph $G_{\h}$ has the following form

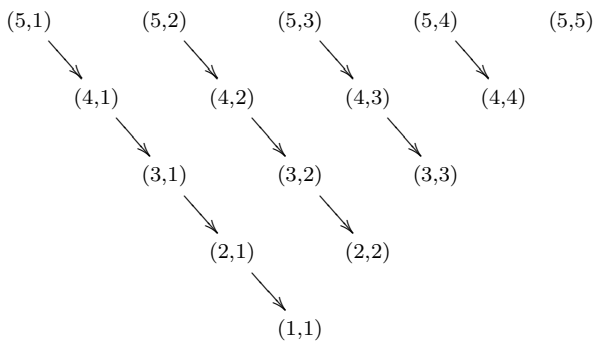
\begin{center}
\begin{tabular}{c c}
\xymatrixrowsep{0.5cm}
\xymatrixcolsep{0.1cm}\xymatrix @C=0.1em { \scriptstyle{(5,1)}\ar[rd]& &\scriptstyle{(5,2)} \ar[rd]   & &\scriptstyle{(5,3)}\ar[rd] &   &\scriptstyle{(5,4)} \ar[rd] &   &\scriptstyle{(5,5)}\\
& \scriptstyle{(4,1)}\ar[rd]& &\scriptstyle{(4,2)} \ar[rd]   & &\scriptstyle{(4,3)}\ar[rd] &   &\scriptstyle{(4,4)} \\
 & & \scriptstyle{(3,1)}  \ar[rd]   & &\scriptstyle{(3,2)}\ar[rd]  & &\scriptstyle{(3,3)}  \\
&    &     &\scriptstyle{(2,1)} \ar[rd] & &\scriptstyle{(2,2)}  & \\
&  &    & &\scriptstyle{(1,1)} & &\\
}\\

\end{tabular}

\end{center}
\begin{remark}
A direct computation shows that for any relation graph $G$ containing $G_{\mathfrak{h}}$ as a subgraph, and any $G$-realization $T(L)$, the module $V_{G}(T(L))$ is a highest weight module with respect to $\mathfrak{h}$.
\end{remark}

\begin{definition} By $\Sigma$ we denote the set of directed graphs $G$ with vertices in $\mathfrak{V}$, such that for any $1\leq s< r\leq n$, there is an arrow from $(r,s)$ to $(r-1,s)$, or there is an arrow from $(r-1,s)$ to $(r,s)$, and no other arrows are allowed in $G$. 
\end{definition}
For any $1\leq s < r\leq n$ we define $A_{(r,s)}:\Sigma\to \{1,-1\}$ given by 

$$
A_{(r,s)}(G):=\begin{cases}
   1 & \text{ if there is an arrow in $G$ from $(r,s)$ to $(r-1,s)$},\\
     -1 & \text{ if there is an arrow in $G$ from $(r-1,s)$ to $(r,s)$}.
\end{cases}$$

In what follows, we are going to construct a subset $\widetilde{\Sigma}$ of $\Sigma$ that parametrizes all possible graphs associated with the relations satisfied by highest weight vectors with respect to all possible Cartan subalgebras. Moreover, we describe an action of $S_n$ on $\Sigma$ such that $\sigma(G_{\mathfrak{h}})$ corresponds with the relations satisfied by a highest weight vector with respect to $\sigma(\mathfrak{h})$.\\

\begin{proposition}\label{PropTech01} There is a bijective correspondence between the set $\Sigma$ and the set of quase-partitions of $\Delta$ given by $G\mapsto \mathcal{Q}_G:=\{\alpha_{ij}\ |\ A_{(j,i)}(G) = 1\}\cup\{\alpha_{ji}\ |\ A_{(j,i)}(G) = -1\}$.
     
\end{proposition}
\begin{proof}  Let $\mathcal{Q}$ be a quase-partition of $\Delta$. Define $G$ to be the graph in $\Sigma$ such that $A_{(j,i)}(G)=1$ if $\alpha_{ij} \in  \mathcal{Q}$ and $A_{(j,i)}(G)=-1$ if $-\alpha_{ij} \in  \mathcal{Q}$. As $|\mathcal{Q}| = |\Delta|/2$ we fulfill all the arrows for the graph $G$. Conversely, for $G \in \Sigma$ consider $\mathcal{Q}_G$ as above. Clearly $|\Delta| = 2|\mathcal{Q}_G|$. Now, if $\alpha_{ij}\in\mathcal{Q}_G$ we have two cases. First $i<j$, here $A_{(j,i)(G)}=1$ and if $-\alpha_{ij} = \alpha_{ji} \in \mathcal{Q}_G$ we have $A_{(j,i)(G)}=-1$, which is a contradiction. Similarly, if we have $i>j$. So, $\mathcal{Q}_G$ is a quase-partition of $\Delta.$
\end{proof}

\begin{corollary}\label{GhDelta+}
    The set of arrows of $G_{\h}$ is in bijective correspondence with the set of positive roots $\Delta^{+}$.
\end{corollary}

\begin{proof} By Proposition \ref{PropTech01}, $G_\h$ corresponds with $\mathcal{Q}_{G_\h}= \Delta^+$.
   
\end{proof}

The previous results allow us to do the following identification. Let $G \in \Sigma$ arbitrary and let $1\leq s < r \leq n$, if $A_{(r,s)}(G) = 1$ the arrow $(r,s) \to (r-1,s)$ corresponds with the positive root $\alpha_{sr}$. On the other hand, if $A_{(r,s)}(G) = -1$, the arrow $(r-1,s) \to (r,s)$ corresponds with the negative root $\alpha_{rs}$. \\

For any $\sigma\in S_n$, the correspondence from Proposition \ref{PropTech01} guaranties the existence of a unique graph $G_{\sigma}$ in $\Sigma$ such that $\mathcal{Q}_{G_{\sigma}}=\sigma(\Delta^+)$. Therefore, the action of $S_n$ on closed partitions induces an action of $S_n$ on $\Sigma$. With this in mind, we have the following definition.

\begin{definition}\label{permutation of a graph}  For $\sigma \in S_n$, we define  $\sigma(G_{\mathfrak{h}}):=G_{\sigma}$. The orbit of $G_\h$ under the action of $S_n$ is denoted by $\tilde{\Sigma}$.
\end{definition}

\begin{lemma}\label{PropTech02}
The correspondence $G\mapsto \mathcal{Q}_G$ defines a bijection between the set $\tilde{\Sigma}$ and the set of closed partitions of $\Delta$.
\end{lemma}

\begin{proof}
Follows directly from the definition of $\tilde{\Sigma}$ and the fact that closed partitions are conjugate by the Weyl group action.
\end{proof} 

\begin{theorem}\label{s-act-Gh} If $\sigma \in S_n$,
then the orientation of the arrows of $\sigma(G_{\h})$ is given by
$$A_{(r,s)}(G_{\sigma})=\begin{cases}
1, & \text{ if } \sigma^{-1}(r)>\sigma^{-1}(s),\\
-1, & \text{ if } \sigma^{-1}(r)<\sigma^{-1}(s),
\end{cases}$$
for any $1\leq s < r \leq n$.
\end{theorem}

\begin{proof}
By the definition of $G_{\sigma}$ we have $A_{(r,s)}(G_{\sigma})=1$, if and only if $\alpha_{sr} \in \sigma(\Delta^+)$, if and only if $\alpha_{sr} = \sigma(\alpha_{ij}) = \alpha_{\sigma(i)\sigma(j)} $ with $1\leq i<j \leq n$, if and only if $\sigma^{-1}(s)<\sigma^{-1}(r)$.
\end{proof}

\begin{example}\label{w0 of G}
Let $\omega_0$ be the longest element in $S_n$, that is $\omega_{0}(i)=n-i+1$ for any $i=1,\ldots, n$. Hence  $i>j$ implies $\omega_{0}(i)<\omega_{0}(j)$ for any $i,j\in{1,\ldots,n}$, and so $\omega_0(G_{\mathfrak{h}})$ is obtained from $G_{\mathfrak{h}}$ by changing the orientation of any arrow, that is, $A_{(rs)}(\omega_0(G_{\mathfrak{h}}))=-1$ for any $r,s$.
\end{example} 

The previous theorem describes explicitly the graph $G_{\sigma}$, however, in the upcoming sections, it will be useful to provide a step by step construction of the graph $G_{\sigma}$ depending on a given presentation of the permutation $\sigma$ as product of simple transpositions. 

\begin{corollary}\label{co: s-act-Gh} Let $G$ be any graph in $\tilde{\Sigma}$ and let $s_k$ be a simple reflection in $S_{n}$. The graph $s_{k}(G)$ is obtained from $G$ as follows:
\begin{itemize}
\item[(i)] $A_{(k+1,k)}(s_{k}(G))=-A_{(k+1,k)}(G)$,
\item[(ii)] $A_{(r,k)}(s_{k}(G))=A_{(r,k+1)}(G)$ for $k+2\leq r\leq n$,
\item[(iii)] $A_{(r,k+1)}(s_{k}(G))=A_{(r,k)}(G)$ for $k+2\leq r\leq n$, 
\item[(iv)] $A_{(k+1,r)}(s_{k}(G))=A_{(k,r)}(G)$ for $1\leq r\leq k-1$,
\item[(v)] $A_{(k,r)}(s_{k}(G))=A_{(k+1,r)}(G)$ for $1\leq r\leq k-1$,
\item[(vi)] $A_{(i,j)}(s_{k}(G))=A_{(i,j)}(G)$, if $i,j\notin\{k,k+1\}$.
\end{itemize}
\end{corollary}

\subsection{Tableaux associated with graphs in the $S_n$-orbit of $G_{\mathfrak{h}}$}
In this section, we explicitly construct Gelfand-Tsetlin tableaux associated with graphs in $\tilde{\Sigma}$ (i.e. the $S_n$-orbit of $G_{\mathfrak{h}}$ in $\Sigma$), and state some technical lemmas relative to the $S_n$ action.  

\begin{definition}\label{twisted-tableux} 
Associated with $X\in\mathbb{C}^n$, and $\sigma\in S_n$, let $T(Y):=T_{\sigma}(X)$ be the Gelfand-Tsetlin tableau constructed recursively as follows:
    \begin{itemize}
 \item[(i)] Row $n$ is given by $y_{n,i}=x_{\sigma^{-1}(i)}$.
 \item[(ii)] Once row $k$ is constructed, row $k-1$ is given by $$y_{k-1,i}=\begin{cases}
     y_{ki},& \text{if, there is an arrow from $(k,i)$ to $(k-1,i)$ in $G_{\sigma}$,}\\

     y_{ki}+1,& \text{if, there is an arrow from $(k-1,i)$ to $(k,i)$ in $G_{\sigma}$.
     }
 \end{cases}$$
    \end{itemize}
\end{definition}

\begin{lemma}\label{y_ij}
    Let $X\in\mathbb{C}^n$ and $\sigma \in S_n$. The tableau $T(Y):=T_{\sigma}(X)$ satisfies the graph $G_{\sigma}$. Moreover, for $1\leq j \leq i<n$, we have 
\begin{equation}\label{y_ij explicit} 
y_{ij} = y_{nj} + \frac{1}{2} \displaystyle \sum_{\ell = i+1}^{n}(1-A_{(\ell, j)}(G_{\sigma})).
\end{equation}

\end{lemma}
\begin{proof} The fact that $T(Y)$ satisfies the graph $G_{\sigma}$ follows directly from the construction of the tableau. Finally, note that $y_{ij}$ is obtained from $y_{i+1,j}$ by adding $1$ or $0$, depending on the sign of $A_{(i+1,j)}(G_{\sigma})$, so identity (\ref{y_ij explicit}) is obtained recursively.
\end{proof} 

The following technical lemmas will be used only for the proof of Theorem \ref{MainThm01}. We leave their proofs for the appendix. \\

For any tableau $T(R)$, we denote by $\Sigma_{t}(R)$ the sum of the entries in row $t$.
 
\begin{lemma}\label{Lemma: sk and tableaux}
Set $\tau\in S_n$ and $s_k$ a simple transposition. If $T(R)=T_{\tau}(X)$ and $T(W)=T_{s_k\circ\tau}(X)$, then

\begin{itemize}
\item[(i)] $\Sigma_{i}(W)=\Sigma_{i}(R)$ for $i\neq k$.

\item[(ii)] $\Sigma_{k}(W)=\Sigma_{k}(R)+ r_{n,k+1}- r_{n,k} + \tau^{-1}(k+1)-\tau^{-1}(k)$.
\end{itemize}

\end{lemma}
\begin{proof}
See appendix \ref{appendix}.
\end{proof}

\begin{lemma}\label{PropTech03} Set $\sigma\in S_n$, $X\in\mathbb{C}^n$, $G:=G_{\sigma}$, and $T(L):=T_{\sigma}(X)$. Given $r<s$, and $\{i_r,i_{r+1},\ldots,i_{s-1}\}$ with $1\leq i_t\leq t$ and $r\leq t <s$, we have:
\begin{itemize}
\item[(i)] If $A_{(s,r)}(G)=1$, then  $T(L+\delta^{r,i_r}+\ldots+\delta^{s-1,i_{s-1}})$ does not  satisfy  $G$.
\item[(ii)] If $A_{(s,r)}(G)=-1$, then $T(L-\delta^{r,i_r}-\ldots-\delta^{s-1,i_{s-1}})$ does not  satisfy $G$.
\end{itemize}

\end{lemma}
\begin{proof}
See appendix \ref{appendix}.
\end{proof}

\subsection{Highest weight modules with respect to different Borel subalgebras}

Recall that $G(T(L))$ denotes the graph associated with the tableau $T(L)$ (see  Definition \ref{Graph associated with a tableau}). For any $\lambda \in \mathfrak{h}^*$, let $\bar{\lambda} := \lambda + \wp$ where $\wp:=(0,-1,-2,\ldots,-n+1)$. Finally, for any tableau $T(R)$, we write $\omega_{t}(R)=t-1+\Sigma_t(R)-\Sigma_{t-1}(R)$. 

\begin{theorem}\label{MainThm01}
Set $\lambda\in\mathfrak{h}^{*}$ and $\sigma\in S_n$. If the graph $G=G(T_{\sigma}(\bar{\lambda}))$ is a relation graph, then  $T_{\sigma}(\bar{\lambda})$ is a highest weight vector of weight $\lambda$ with respect to the triangular decomposition of $\mathfrak{g}$ induced by $\sigma(\mathfrak{h})$. In particular, $V_{G}(T_{\sigma}(\bar{\lambda}))$ is a highest weight module  of highest weight $\lambda$ with respect to the Cartan subalgebra $\sigma(\mathfrak{h})$.
\end{theorem}
\begin{proof}

We first prove that  $T_{\sigma}(\bar{\lambda})$  is a tableau of weight $\lambda$ with respect to $\sigma(\mathfrak{h})$. 
first of all, the statement is true when $\sigma$ is the identity. Indeed, $T(R):=T_{e}(\bar{\lambda})$ is the tableau with entries $r_{ij}=\lambda_j-j+1$, and $
\omega_{k}(T(R))=k-1+\sum\limits_{i=1}^{k}(\lambda_{i}-i+1)-\sum\limits_{i=1}^{k-1}(\lambda_{i}-i+1)=\lambda_k$.\\

In order to prove the statement for any $\sigma$, we consider $\tau\in S_n$ arbitrary and show that whenever $T(R):=T_{\tau}(\bar{\lambda})$  has weight $\lambda$ with respect to $\tau(\mathfrak{h})$, we also have that  $T(W):=T_{s_k\circ \tau}(\bar{\lambda})$ has weight $\lambda$ with respect to $(s_k\circ \tau)(\mathfrak{h})$ for any simple transposition $s_k$.

\begin{itemize}
\item[(i)] Suppose first that $\tau(i)\notin\{k,k+1\}$. Then 
\begin{align*}
\omega_{s_k\circ \tau(i)}(W)=&\omega_{ \tau(i)}(W)
=\tau(i)-1+\Sigma_{\tau(i)}(W)-\Sigma_{\tau(i)-1}(W)\\
=&\tau(i)-1+\Sigma_{\tau(i)}(R)-\Sigma_{\tau(i)-1}(R)\\
=&\omega_{\tau(i)}(R)\\
=&\lambda_i.
\end{align*}
Here we used Lemma \ref{Lemma: sk and tableaux} (i).
\item[(ii)] Suppose that $\tau(i)=k$ and $\tau(j)=k+1$. Then $\omega_{s_k\circ \tau(i)}(W)$ is equal to
\begin{align*} 
\omega_{k+1}(W)=&k+\Sigma_{k+1}(W)-\Sigma_{k}(W)\\
=&k+\Sigma_{k+1}(R)-(\Sigma_{k}(R)+\lambda_{\tau^{-1}(k+1)}-\lambda_{\tau^{-1}(k)})\\
=&\omega_{k+1}(R)-(\lambda_{j}-\lambda_{i})\\
=&\omega_{\tau(j)}(R)-(\lambda_{j}-\lambda_{i})\\
=&\lambda_j-(\lambda_{j}-\lambda_{i})\\
=&\lambda_i.
\end{align*}

Here we used Lemma \ref{Lemma: sk and tableaux} (ii).
\item[(iii)] The case $\tau(i)=k+1$ and $k=\tau(j)$ is analogous to case (ii). 
\end{itemize}

To prove that the tableau  $T(L):=T_{\sigma}(\bar{\lambda})$ is a highest weight vector with respect to $\sigma(\mathfrak{h})$, it is enough to prove that for any $r>s$,
\begin{itemize}
\item[(a)] $E_{r,s}(T(L))=0$, whenever $A_{(r,s)}(G_{\sigma})=1$.
\item[(b)] $E_{s,r}(T(L))=0$, whenever  $A_{(r,s)}(G_{\sigma})=-1$.
\end{itemize}

Indeed, from Remark \ref{general formulas} we have $E_{\ell m} (T(L))= \sum\limits_{\tau \in \Phi_{\ell m}} e_{\ell m} (\tau (L)) T(L+\tau(\varepsilon_{\ell m}))$. To prove (a) we use 
Lemma \ref{PropTech03} (i) to conclude that  $T(L+\tau(\varepsilon_{\ell m}))$ does not satisfy $G$ for any $\tau$. Analogously, to prove (b) we use Lemma \ref{PropTech03} (ii) to show that $T(L+\tau(\varepsilon_{\ell m}))$ does not satisfy $G$ for any $\tau$.
\end{proof}

The following definition is motivated by Theorem \ref{MainThm01}.

\begin{definition}\label{def lambda-sigma-adm}
Let $\sigma\in S_{n}$ and $\lambda\in \mathfrak{h}^*$. We say that $\lambda$ is a \emph{$\sigma$-relation weight} if the graph $G(T_{\sigma}(\bar{\lambda}))$ is a relation graph.
\end{definition}

For $\lambda\in\mathfrak{h}^{*}$, we denote by $M(\lambda)$ the Verma module of highest weight $\lambda$, and by $L(\lambda)$ the simple quotient of $M(\lambda)$ by its unique maximal submodule. According to Definition \ref{def lambda-sigma-adm}, a weight $\lambda$ will be called \emph{$id$-relation weight} if the highest weight module $L(\lambda)$ is a relation Gelfand-Tsetlin module with respect to the standard Gelfand-Tsetlin subalgebra. In \cite{FMR21}, a characterization of $id$-relation weights was provided. Namely,

\begin{proposition}\label{Th: id-relation}A weight $\lambda=(\lambda_1, \dots, \lambda_n)$ is an $id$-relation weight if and only if one of the following conditions holds:
\begin{itemize}
\item[(i)] $\lambda_i-\lambda_j\notin \mathbb{Z}_{\leq i-j} $ for all $1\leq i< j< n$.
\item[(ii)] There exist unique $i,j$ with $1\leq i< j< n$ such that:
    \begin{itemize}
    \item[(a)] $\lambda_r-\lambda_{r+1}\in \mathbb Z_{\geq 0}$ for each $r\geq j$,
    \item[(b)]$\lambda_r-\lambda_s\notin \mathbb{Z}_{\leq r-s} $  for each $r\neq i$ and $s\geq j$,
    \item[(c)] $\lambda_n-\lambda_i\in \mathbb{Z}_{\geq n-i}$. 
     \end{itemize}
\end{itemize}
\end{proposition}

\begin{proof}
See \cite[Theorem 4.8]{FMR21}.
\end{proof}

\section{Main results and applications}\label{S:Main results and applications}

In this section, we present conditions for a weight $\lambda$ to be a $\sigma$-relation weight. In particular, we provide realizations of a simple highest weight module as a relation module with a twisted action of the Gelfand-Tsetlin formulas. We also explore conditions for a weight $\lambda$ to be a relation weight and provide an inductive method to construct relation weights. Finally, we construct non-highest weight modules using localization functors.\\

Recall that for any $\lambda \in \mathfrak{h}^*$, we set $\bar{\lambda} := \lambda + \wp$ where $\wp:=(0,-1,-2,\ldots,-n+1)$.

\subsection{\texorpdfstring{$\sigma$}{sigma}-relation modules and weights} 
Each element $\sigma$ of $S_n$ induces a natural automorphism $\varphi_{\sigma}:\mf{gl}_n\to \mf{gl}_n$ defined on generators by $\varphi_{\sigma}(E_{ij}):=E_{\sigma(i),\sigma(j)}.$
Denote by $\widetilde{\varphi}_{\sigma}$ the extension of $\varphi_{\sigma}$ to $U(\mf{gl}_n)$, and denote by $B_{\sigma}$  the image of $B\subseteq U(\mf{gl}_n)$ under $\widetilde{\varphi}_{\sigma}$. For any $\mathfrak{gl}_n$-module $M$ and $\sigma\in S_n$, we denote by $M^{\sigma}$ the module $M$ twisted by the action of the automorphism $\varphi_{\sigma}$ (i.e. the vector space $M$ with the action of $U(\mathfrak{gl}_n)$ given by $x\cdot_{\sigma}m=\tilde{\varphi}_\sigma(x)\cdot m$).

\begin{definition}\label{def sigma-adm}
Given $\sigma\in S_{n}$, a $\mathfrak{gl}_{n}$-module $M$ will be called \emph{$\sigma$-relation module} if $M$ is isomorphic to  $N^{\sigma}$ for some relation module $N$ with respect to the standard Gelfand-Tsetlin subalgebra $\Gamma$. 
\end{definition}

\begin{theorem}\label{MainThm02}
Set $\lambda\in\mathfrak{h}^*$, $\sigma\in S_n$, and $G = G(T_{\sigma}(\bar{\lambda}))$. If $\lambda$ is  a $\sigma$-relation weight, then $L(\lambda)$ is a $\sigma$-relation module, more precisely
$$L(\lambda)\simeq [V_{G}(T_{\sigma}(\bar{\lambda}))]^{\sigma}.$$
\end{theorem}

\begin{proof}
Follows directly from Theorem \ref{MainThm01}.
\end{proof}

\begin{example} \label{generic-pointed}
Any $\sigma$-relation module has all its $\Gamma_\sigma$-multiplicities bounded by one (see Theorem \ref{mult_rel_mods}), however, the converse is not necessarily true. Indeed, let $\lambda=\left(-\frac{1}{6},-\frac{2}{3}, \frac{5}{6} \right)$, the tableaux $T_{id}(\bar{\lambda})$, and $T_{s_2}(\bar{\lambda})$ are, respectively, 
 $$ {\tiny \xymatrix@R=1pt@C=1pt{-\dfrac{1}{6}&&-\dfrac{5}{3}&&-\dfrac{7}{6}\\&-\dfrac{1}{6}&&-\dfrac{5}{3}&\\&&-\dfrac{1}{6}&&}\hspace{2cm} \xymatrix@R=1pt@C=1pt{-\dfrac{1}{6}&&-\dfrac{7}{6}&&-\dfrac{5}{3}\\&-\dfrac{1}{6}&&-\dfrac{1}{6}&\\&&-\dfrac{1}{6}&&} } $$

the module $L(\lambda)$ is a relation module, as the graph associated with  $T_{id}(\bar{\lambda})$ is 
 
 \begin{center}
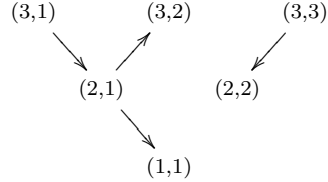
\begin{tabular}{c c }
\xymatrixrowsep{0.5cm}
\xymatrixcolsep{0.1cm}\xymatrix @C=0.1em {
  &\scriptstyle{(3,1)}   \ar[dr] & &\scriptstyle{(3,2)} & & \scriptstyle{(3,3)}   \ar[ld]\\
 & & \scriptstyle{(2,1)} \ar[ur] \ar[dr]   & &\scriptstyle{(2,2)}    \\
    &  &   &\scriptstyle{(1,1)}  &   \\
}
\end{tabular}
\end{center} 
 
Moreover, all its weight multiplicities are $1$ (see \cite[Section 7.2, Case (G11)]{FGR21}), and consequently, the action of $\Gamma_\sigma$ is diagonalizable for any $\sigma\in S_3$. However, $L(\lambda)$ is not $s_2$-relation since the graph associated with the tableau $T_{s_2}(\bar{\lambda})$ given by 

 \begin{center}
\begin{tabular}{c c }
\xymatrixrowsep{0.5cm}
\xymatrixcolsep{0.1cm}\xymatrix @C=0.1em {
  &\scriptstyle{(3,1)} \ar[dr]  & &\scriptstyle{(3,2)} & & \scriptstyle{(3,3)}   \\
 & & \scriptstyle{(2,1)}  \ar[ur] \ar[dr]  & &\scriptstyle{(2,2)} \ar[ul] \ar[ld]   \\
    &  &   &\scriptstyle{(1,1)}  &   \\
}
\end{tabular}
\end{center} 
which is critical (see Definition \ref{Def: critical}).
\end{example}  

\begin{example} It is possible to have a weight $\lambda$ and permutations $\sigma\neq \tau$, such that $\lambda $ is a $\sigma$-relation weight but is not  a $\tau$-relation weight. Indeed, for $\lambda=(-1,0,1)$, the tableaux $T_{id}(\bar{\lambda})$, and $T_{s_2}(\bar{\lambda})$ are, respectively, 

$$\xymatrix@R=1pt@C=1pt{-1&&-1&&-1\\&-1&&-1&\\&&-1&&}\hspace{2cm}\xymatrix@R=1pt@C=1pt{-1&&-1&&-1\\&-1 &&0&\\&&-1&&}$$
The module $L(\lambda)$ is not a relation module since the graph associated with the tableau $T_{id}(\bar{\lambda})$ is critical. However, $L(\lambda)$ is a $s_2$-relation module since the graph associated with $T_{s_2}(\bar{\lambda})$ is

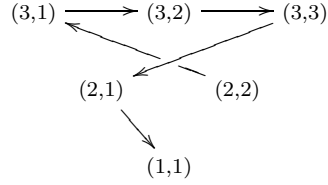
 \begin{center}
\begin{tabular}{c c }
\xymatrixrowsep{0.5cm}
\xymatrixcolsep{0.1cm}\xymatrix @C=0.1em {
  &\scriptstyle{(3,1)} \ar[rr]  & &\scriptstyle{(3,2)} \ar[rr]& & \scriptstyle{(3,3)}  \ar[llld] \\
 & & \scriptstyle{(2,1)}  \ar[dr]   & &\scriptstyle{(2,2)} \ar[ulll] |!{[ul];[l]}\hole \\
    &  &   &\scriptstyle{(1,1)}  &   \\
}
\end{tabular}
\end{center} 

which is a relation graph. 

\end{example}  

\begin{proposition}\label{Relsl3}
If $\lambda=(\lambda_1,\lambda_2, \lambda_3)$ is a $\mathfrak{gl}_3$-weight, there exists $\sigma\in S_3$ such that $\lambda$ is  a $\sigma$-relation weight.
\end{proposition} 

\begin{proof} We do a case by case construction of the desired $\sigma$.

\begin{itemize}
   \item[(i)] Suppose first that $\{\lambda_{1}-\lambda_2,\lambda_{1}-\lambda_3, \lambda_{2}-\lambda_3\}\cap \mathbb{Z}=\emptyset$. In this case, for any $\sigma\in S_3$ the tableau $T_{\sigma}(\bar{\lambda})$ does not have adjoining pairs, and consequently, the graph associated with the tableau is a relation graph. 
 \item[(ii)] Suppose that $\lambda_{i}-\lambda_j\in\mathbb{Z}$ for exactly one pair $1\leq i<j\leq 3$. We provide an explicit $\sigma$ by considering three cases. In all three cases, the tableaux that appear do not have adjoining pairs, and consequently, the associated graph is a relation graph.
 \begin{itemize}
     \item [(1)] Suppose that $\lambda_1-\lambda_3 \in \mathbb{Z}$. In this situation, we consider $\sigma=id$.  Indeed, the tableau $T_{id}(\bar{\lambda})$, and the possible graphs are given by 
$$
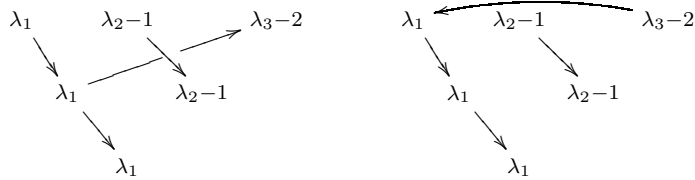
\begin{tabular}{c c}
\xymatrixrowsep{0.5cm}
\xymatrixcolsep{0.1cm}\xymatrix @C=0.1em {
 & & \scriptstyle{\lambda_1}  \ar[rd]   & &\scriptstyle{\lambda_2-1} \ar[dr]  & &\scriptstyle{\lambda_3-2}  \\
&    &     &\scriptstyle{\lambda_1} \ar[rd] \ar[urrr] |!{[u];[rrr]}\hole & &\scriptstyle{\lambda_2-1}  & \\
&  &    & &\scriptstyle{\lambda_1} & &\\
}
\end{tabular}
\begin{tabular}{c c}
\xymatrixrowsep{0.5cm}
\xymatrixcolsep{0.1cm}\xymatrix @C=0.1em {
 & & \scriptstyle{\lambda_1}  \ar[dr]   & &\scriptstyle{\lambda_2-1}\ar[rd]   & &\scriptstyle{\lambda_3-2} \ar@/_/
 [llll]   \\
&    &     &\scriptstyle{\lambda_1}  \ar[rd] & &\scriptstyle{\lambda_2-1}& \\
&  &    & &\scriptstyle{\lambda_1}  & &\\
}
\end{tabular}
$$

\item [(2)] Suppose that $ \lambda_2-\lambda_3 \in \mathbb{Z}$. In this situation, we consider $\sigma=id$.  Indeed, the tableau $T_{id}(\bar{\lambda})$, and the possible graphs are given by 
$$
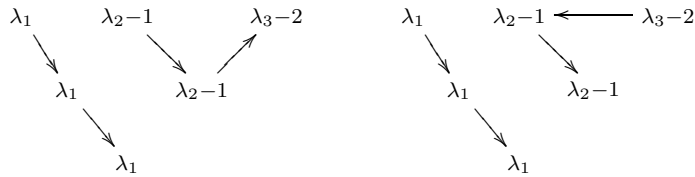
\begin{tabular}{c c}
\xymatrixrowsep{0.5cm}
\xymatrixcolsep{0.1cm}\xymatrix @C=0.1em {
 & & \scriptstyle{\lambda_1}  \ar[dr]  & &\scriptstyle{\lambda_2-1} \ar[dr] & &\scriptstyle{\lambda_3-2}  \\
&    &     &\scriptstyle{\lambda_1} \ar[rd] & &\scriptstyle{\lambda_2-1}\ar[ru]  & \\
&  &    & &\scriptstyle{\lambda_1} & &\\
}
\end{tabular}
\begin{tabular}{c c}
\xymatrixrowsep{0.5cm}
\xymatrixcolsep{0.1cm}\xymatrix @C=0.1em {
 & & \scriptstyle{\lambda_1}  \ar[rd]   & &\scriptstyle{\lambda_2-1}\ar[dr] & &\scriptstyle{\lambda_3-2} \ar[ll] \\
&    &     &\scriptstyle{\lambda_1}  \ar[rd] & &\scriptstyle{\lambda_2-1}  & \\
&  &    & &\scriptstyle{\lambda_1} & &\\
}
\end{tabular}
$$

\item [(3)] Suppose that $\lambda_1-\lambda_2 \in \mathbb{Z}$. In this situation we consider $\sigma=s_2$. The tableau $T_{s_2}(\bar{\lambda})$ and the possible graphs are given by 
     $$
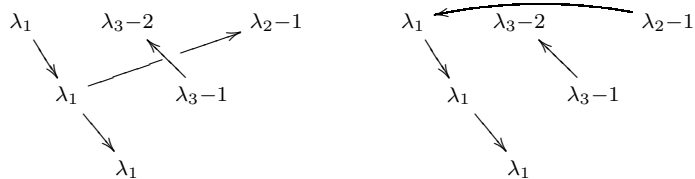
\begin{tabular}{c c}
\xymatrixrowsep{0.5cm}
\xymatrixcolsep{0.1cm}\xymatrix @C=0.1em {
 & & \scriptstyle{\lambda_1}  \ar[rd]   & &\scriptstyle{\lambda_3-2} & &\scriptstyle{\lambda_2-1} \\
&    &     &\scriptstyle{\lambda_1} \ar[rd]  \ar[urrr] |!{[u];[rrr]}\hole  & &\scriptstyle{\lambda_3-1}  \ar[lu] & \\
&  &    & &\scriptstyle{\lambda_1} & &\\
}
\end{tabular}
\begin{tabular}{c c}
\xymatrixrowsep{0.5cm}
\xymatrixcolsep{0.1cm}\xymatrix @C=0.1em {
 & & \scriptstyle{\lambda_1}  \ar[dr]   & &\scriptstyle{\lambda_3-2}   & &\scriptstyle{\lambda_2-1}  \ar@/_/[llll]   \\
&    &     &\scriptstyle{\lambda_1}  \ar[rd] & &\scriptstyle{\lambda_3-1} \ar[ul] & \\
&  &    & &\scriptstyle{\lambda_1}  & &\\
}
\end{tabular}
$$ 
\end{itemize}
   
\item[(iii)] Suppose that $\lambda_{i}-\lambda_j\in\mathbb{Z}$ for any $1\leq i<j\leq 3$. We provide an explicit $\sigma$ by considering three cases. In all cases, the graphs that appear satisfy the $\diamond$-condition; therefore, Theorem \ref{relation mod} implies that $\lambda$ is a $\sigma$-relation weight.
\begin{itemize}
\item[(a)]Suppose that $ \lambda_1-\lambda_2 \in \mathbb{Z}_{\geq 0}$. In this case, we consider $\sigma=id$. Indeed, the tableau $T_{id}(\bar{\lambda})$ and the possible graphs are given by 
$$
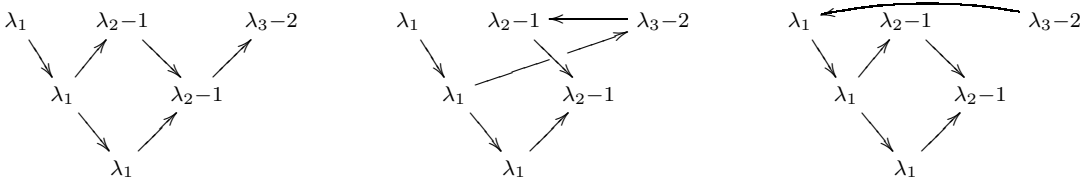
\begin{tabular}{c c}
\xymatrixrowsep{0.5cm}
\xymatrixcolsep{0.1cm}\xymatrix @C=0.1em {
 & & \scriptstyle{\lambda_1}  \ar[rd]   & &\scriptstyle{\lambda_2-1} \ar[dr] & &\scriptstyle{\lambda_3-2}  \\
&    &     &\scriptstyle{\lambda_1} \ar[rd] \ar[ur]& &\scriptstyle{\lambda_2-1}\ar[ru]  & \\
&  &    & &\scriptstyle{\lambda_1} \ar[ur]& &\\
}
\end{tabular}
\begin{tabular}{c c}
\xymatrixrowsep{0.5cm}
\xymatrixcolsep{0.1cm}\xymatrix @C=0.1em {
 & & \scriptstyle{\lambda_1}  \ar[rd]   & &\scriptstyle{\lambda_2-1}\ar[dr] & &\scriptstyle{\lambda_3-2} \ar[ll] \\
&    &     &\scriptstyle{\lambda_1}\ar[rrru] |!{[u];[rrr]}\hole \ar[rd] & &\scriptstyle{\lambda_2-1}  & \\
&  &    & &\scriptstyle{\lambda_1}\ar[ru]& &\\
}
\end{tabular}
\begin{tabular}{c c}
\xymatrixrowsep{0.5cm}
\xymatrixcolsep{0.1cm}\xymatrix @C=0.1em {
 & & \scriptstyle{\lambda_1}  \ar[dr]   & &\scriptstyle{\lambda_2-1}\ar[rd]   & &\scriptstyle{\lambda_3-2} \ar@/_/[llll]   \\
&    &     &\scriptstyle{\lambda_1}\ar[ur]  \ar[rd] & &\scriptstyle{\lambda_2-1}& \\
&  &    & &\scriptstyle{\lambda_1}\ar[ru]  & &\\
}
\end{tabular}
$$

\item[(b)] Suppose that $\lambda_1-\lambda_2 \in \mathbb{Z}_{< 0}$, and 
$ \lambda_2-\lambda_3 \in \mathbb{Z}_{< 0}$. In this situation, we consider $\sigma=s_2$. In this case, the tableau $T_{s_2}(\bar{\lambda})$ and the associated graph are 

$$
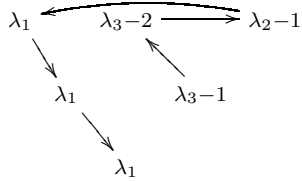
\begin{tabular}{c c}
\xymatrixrowsep{0.5cm}
\xymatrixcolsep{0.1cm}\xymatrix @C=0.1em {
 & & \scriptstyle{\lambda_1}  \ar[dr]   & &\scriptstyle{\lambda_3-2} \ar[rr] & &\scriptstyle{\lambda_2-1} \ar@/_/[llll]   \\
&    &     &\scriptstyle{\lambda_1} \ar[dr]  & &\scriptstyle{\lambda_3-1}  \ar[ul] & \\
&  &    & &\scriptstyle{\lambda_1}& &\\
}
\end{tabular}
$$

\item[(c)] Suppose that $\lambda_1-\lambda_2 \in \mathbb{Z}_{< 0}$ and $ \lambda_2-\lambda_3 \in \mathbb{Z}_{\geq 0}$. Here we consider $\sigma=s_1s_2s_1$. Indeed, the tableau $T_{\sigma}(\bar{\lambda})$, and the possible graphs are given by 
$$
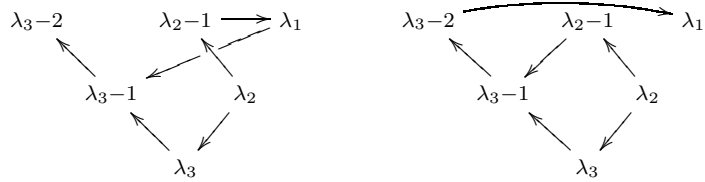
\begin{tabular}{c c}
\xymatrixrowsep{0.5cm}
\xymatrixcolsep{0.1cm}\xymatrix @C=0.1em {
 & & \scriptstyle{\lambda_3-2}   & &\scriptstyle{\lambda_2-1} \ar[rr] & &\scriptstyle{\lambda_1} \ar[ldll] |!{[dd];[ll]}\hole \\
&    &     &\scriptstyle{\lambda_3-1}\ar[lu]    & &\scriptstyle{\lambda_2} \ar[ld] \ar[lu] & \\
&  &    & &\scriptstyle{\lambda_3}\ar[lu] & &\\
}
\end{tabular}
\begin{tabular}{c c}
\xymatrixrowsep{0.5cm}
\xymatrixcolsep{0.1cm}\xymatrix @C=0.1em {
 & & \scriptstyle{\lambda_3-2} \ar@/^/[rrrr]    & &\scriptstyle{\lambda_2-1} \ar[dl]  & &\scriptstyle{\lambda_1}     \\
&    &     &\scriptstyle{\lambda_3-1}\  \ar[ul] & &\scriptstyle{\lambda_2}  \ar[dl] \ar[lu]& \\
&  &    & &\scriptstyle{\lambda_3} \ar[lu] & &\\
}
\end{tabular}
$$

\end{itemize}

 \end{itemize} 
\end{proof}

\begin{corollary}\label{Rel_n<3}
Let $\lambda$ be any $\mathfrak{gl}_n$-weight with $n\leq 3$. Then there exists $\sigma\in S_n$ such that $\lambda$ is  a $\sigma$-relation weight.
\end{corollary}

\begin{proof}
Follows from Proposition \ref{Relsl3} and the fact that for $n\leq 2$, the graphs do not have adjoining pairs and, consequently, are relation graphs.
\end{proof}

\subsection{Applications} We start by presenting a generalization of \cite[Theorem 4.8]{FMK23}.

\begin{proposition}\label{rel_any_perm}
    Let  $\lambda=(\lambda_1, \lambda_2, \dots, \lambda_n)$ be a $\mathfrak{gl}_n$-weight. In the following cases, $\lambda$ is  a $\sigma$-relation weight for all $\sigma\in S_n$:

    \begin{itemize}
\item[(i)] If $\lambda_i-\lambda_j \notin \mathbb Z$ for all $1\leq i < j\leq n$,

\item[(ii)] If $\lambda_i-\lambda_j \in \mathbb Z_{\geq 0}$ for all $1\leq i < j\leq n$.
    \end{itemize}
\end{proposition}

\begin{proof}
In case (i) $(\lambda_i-i+1)-(\lambda_j-j+1) \notin \mathbb Z$ for all $1\leq i < j\leq n$, then the graph $G=G(T_{\sigma}(\bar{\lambda}))$ does not have adjoining pairs; therefore, $V_{G}(T_{\sigma}(\bar{\lambda}))$ is a simple relation $\Gamma$-module for all $\sigma$. In case (ii) $\lambda$ is an integral dominant $\mathfrak{gl}_n$-weight, then $L(\lambda)$ is a finite-dimensional $\mathfrak{gl}_n$-module.

\end{proof}

\begin{proposition}
Let $\lambda=(\lambda_1, \lambda_2, \dots, \lambda_n)$ be a $\mathfrak{gl}_n$-weight. If  $\lambda_1-\lambda_n \notin \mathbb{Z}_{\geq 1-n}$, $\lambda_r-\lambda_s \notin \mathbb{Z}$ for all $1\leq s \leq n$, and $r\notin\{ 1,s,n\}$. Then the Verma module $M(\lambda)$ is a $\sigma$-relation module for all $\sigma\in S_n$.
\end{proposition}

\begin{proof}
If $\lambda_1-\lambda_n \notin \mathbb{Z}$, the statement follows from Proposition \ref{rel_any_perm}. Now suppose that $\lambda_1-\lambda_n \in \mathbb{Z}_{< 1-n}$ and define $X=\sigma(\lambda +\wp)$, then there exists $1\leq i<j \leq n$ such that $\sigma(1)=i$ and $\sigma(n)=j$ or $\sigma(n)=i$ and $\sigma(1)=j$. Hence,  $x_i-x_j\in \mathbb{Z}_{>0}$ or $x_i-x_j\in \mathbb{Z}_{<0}$ and $x_r-x_s\notin \mathbb{Z}$ for any $1\leq r, s\leq n$ such that $r\neq i, j, s$. Consider the case  $\sigma(n)=i$ and $\sigma(1)=j$ (the other is similar) in this situation, the tableau $T_\sigma(X)$ has entries with integral differences given by $v_{ki}$ and $v_{kj}$ for all $k=j, j+1, \dots ,n$, on the other hand, by Theorem \ref{s-act-Gh} $A_{(k, i)}(G_\sigma)=1$ and $A_{(k, j)}(G_{\sigma})=-1$ for all $k>j$, thus the set of the adjoining pairs in $G=G_\sigma \cup \{(n,i) \longrightarrow (n,j)  \}$ is $\{((k,i), (k,j) )\ |\ k =j, j+1, \dots , n-1 \}$ and $G$ satisfies the $\diamond$-condition. Therefore, $G$ is the graph associated to $T_\sigma(X)$ and the statement follows from Theorem \ref{MainThm02}.
\end{proof}

Given any graph $G$ with set of vertices $\mathfrak{V}$, and $1\leq r<s\leq n$, we denote by $G_{(r,s)}$ the subgraph of $G$ obtained from $G$ by restriction to the set of vertices $\{(i,j)\ |\ r\leq j\leq i\leq s\}$. Also, for any $A=\{a_1,\ldots,a_k\}\subseteq\{1,\ldots,n\}$ with $a_1<a_2<\cdots<a_k$, we will write $\wp_{A}:=(1-a_1,1-a_2,\ldots,1-a_k)$, and  $\wp_{(k)}:=\wp_{\{1,\ldots,k\}}$.

\begin{theorem}\label{Th:admgrphs}
Let $\lambda=(\lambda_1,\dots, \lambda_n)$ be a $\mathfrak{gl}_n$-weight, and $k \leq n$ such that 

\begin{itemize}
    \item[(i)] There exists $A: = \{a_1, \ldots, a_k\} \subset \{1,\ldots,n\}$ such that $\lambda_i-\lambda_j\in\mathbb{Z}$ with $i\neq j$ implies $i,j\in A$;
    \item[(ii)] There exists $\tau\in S_k$ such that the $\mathfrak{gl}_k$-weight  $\mu_A:=(\lambda_{a_1}, \lambda_{a_2}, \ldots, \lambda_{a_k})+\wp_{A}-\wp_{(k)}$ is  a $\tau$-relation weight.
    \end{itemize}

    Then there exist at least $p (n-k)!$ permutations $\sigma \in S_n$ such that $\lambda$ is a $\sigma$-relation weight, where $p=\#\{ \tau \in S_k \ |\ \mu_A \text{ is a $\tau$-relation weight} \}$.

\end{theorem}

\begin{proof}
Consider the weights $X:=\lambda+\wp_{(n)}=(\lambda_1, \dots, \lambda_i-i+1, \dots, \lambda_n-n+1)$, and 
$Y:=(\lambda_{a_1}, \lambda_{a_2}, \ldots, \lambda_{a_k})+\wp_{A}=(\lambda_{a_1}-a_1+1, \lambda_{a_2}-a_2+1, \dots, \lambda_{a_i}-a_i+1, \dots,\lambda_{a_k}-a_k+1)$.

Let $\tau$ be as in condition (ii), then the graph $G_{A}:=G(T_{\tau}(\mu_{A}+\wp_{(k)}))$ is a relation graph. We consider $\sigma \in S_n$ such that $\sigma^{-1}(n-k+i)=a_{\tau^{-1} (i)}$ for any $1\leq i\leq k$. Then $$\sigma X=(\lambda_{\sigma^{-1}(1)}-\sigma^{-1}(1)+1,\dots, \lambda_{\sigma^{-1}(n-k)}-\sigma^{-1}(n-k)+1,Y_{\tau^{-1}(1)},Y_{\tau^{-1}(2)},\dots, Y_{\tau^{-1}(k)}). $$
By Theorem \ref{s-act-Gh}, $A_{(i,j)}(G_{\tau})=A_{(n-k+i,n-k+j)}(G_{\sigma})$ for any  $1\leq j <i \leq k$.  Moreover, as $\tau Y=\tau(\mu_{A}+\wp_{(k)})$, we have that the graph $\mathcal{G}:=G(T_{\sigma}(X))$ is such that  $\mathcal{G}_{(n-k+1,n)}=G_A$ is a relation graph. Finally, condition (ii) implies that any connected component of the graph $\mathcal{G}\setminus \mathcal{G}_{(n-k+1,n)}$ is a relation graph.
\end{proof}

\begin{corollary}\label{co: admgrphs}
Let $\lambda=(\lambda_1,\dots, \lambda_n)$ be a $\mathfrak{gl}_n$-weight, and $k< n$ such that 
    \begin{itemize}
    \item[(i)] There exists $t\leq n-k$ such that the $\mathfrak{gl}_k$-weight 
    $\mu_t:=(\lambda_t, \lambda_{t+1}\dots, \lambda_{k+t-1})$ is a $\tau$-relation weight for some $\tau\in S_k$. 
    \item[(ii)] $\lambda_i-\lambda_j\in\mathbb{Z}$ with $i\neq j$ implies $t\leq i,j< k+t$.
    \end{itemize}

     Then there exist at least $p (n-k)!$ permutations $\sigma \in S_n$ such that $\lambda$ is a $\sigma$-relation weight, where $p=\#\{ \tau \in S_k \ |\ \mu_t \text{ is a $\tau$-relation weight} \}$.
\end{corollary}

\begin{proof} The weights $\mu_t=(\lambda_t,\dots, \lambda_{k+t-1})$, and $\tilde{\mu_t}:=(\lambda_t-t+1,\dots, \lambda_{k+t-1}-t+1)$ satisfy $G(T_{\nu}(\mu_t))=G(T_{\nu}(\tilde{\mu_t}))$ for any $\nu\in S_k$. In particular, $\mu_t$ is a $\tau$-relation weight  if and only if $\tilde{\mu_t}$ is a $\tau$-relation weight. The statement follows from Theorem \ref{Th:admgrphs} applied to the set $A=\{\ell +t-1\ |\ 1\leq \ell\leq k\}$.
\end{proof}

\begin{corollary}\label{co: lambda-adm}
Let $\lambda=(\lambda_1,\dots, \lambda_n)$ be a $\mathfrak{gl}_n$-weight, and $n_{\lambda}$  the number of permutations $\sigma \in S_n$ such that $\lambda$ is a $\sigma$-relation weight. 
 \begin{itemize}
    \item[(i)] If there exist $k\leq n$ and  $t\leq n-k+1$ such that 
    $\lambda_i- \lambda_{i+1} \in \mathbb{Z}_{\geq 0}$ for all $t\leq i\leq t+k-1$, and $\lambda_i-\lambda_j\in\mathbb{Z}$ with $i\neq j$ implies  $t\leq i,j\leq k+t-1$. Then $n_{\lambda}\geq k!(n-k)!$.

\item[(ii)] If $\lambda_i-\lambda_j \notin \mathbb Z_{\leq i-j}$ for any  $1\leq i < j< n$, then $n_{\lambda}\geq m!(n-m)!$, where $m=\#\{i\ |\ \lambda_i-\lambda_j\in \mathbb Z_{>i-j} \text{ for some } j\}$. 
\item[(iii)] If there exists $A\subseteq \{1,\ldots,n\}$ with $\# A=k\leq 3$ such that  $\lambda_i-\lambda_j\in\mathbb{Z}$ with $i\neq j$ implies  $i,j\in A$, then $n_{\lambda}\geq (n-k)!$.
    \end{itemize}

\end{corollary}

\begin{proof}
The first two statements follow from Proposition \ref{rel_any_perm} and Theorem  \ref{Th:admgrphs}. The third statement follows from Theorem \ref{Th:admgrphs} and Corollary \ref{Rel_n<3}.
\end{proof}

\begin{theorem}\label{Cor_MainThm02}
Let $\lambda=(\lambda_1,\dots, \lambda_n)$ be a  $\mathfrak{gl}_n$-weight. If at least $4$ entries of $\bar{\lambda}$ are equal, then there is no $\sigma\in S_n$ such that $\lambda$ is a $\sigma$-relation weight.
\end{theorem}
\begin{proof}
  Suppose that there exist $\sigma \in S_{n}$, such that $\mathcal{G}:=G(T_{\sigma}(\bar{\lambda}))$ is a relation graph, and $L(\lambda)\simeq [V_{\mathcal{G}}(T_{\sigma}(\bar{\lambda}))]^{\sigma}$. By construction, the top row of the tableau $T(U):=T_{\sigma}(\bar{\lambda})$ is given by $\sigma(\bar{\lambda})$ thus, by hypothesis, it has at least $4$ equal entries. Say $a=u_{n,i}=u_{n,j}=u_{n,k}=u_{n,l}$ with $i<j<k<l$. Again, by construction of $T(U)$, $\{u_{n-1,i}, u_{n-1,j}, u_{n-1,k}\}\subseteq \{a, a+1\}$, which implies that at least $2$ elements of the set $\{u_{n-1,i}, u_{n-1,j}, u_{n-1,k}\}$ are equal, then $T(U)$ is critical. This is a contradiction as $T(U)$ is a $\mathcal{G}$-realization with $\mathcal{G}$ being a relation graph.
\end{proof}

\begin{corollary}
 The $\mathfrak{sl}_n$-module $M(-\wp)$ is a $\sigma$-relation module for some $\sigma\in S_n$ if and only if $n\leq 3$.
 \end{corollary}

\begin{proof}
Follows from Theorem \ref{Cor_MainThm02}, and Proposition \ref{Relsl3}.
\end{proof}

\subsection{Localization of \texorpdfstring{$\sigma$}{sigma}-relation modules}

For a reductive Lie algebra $\mathfrak{g}$ we recall the definition of the localization functor on $\mathfrak{g}$-modules. For more details on this topic, we refer the reader to  \cite{Deo80, Mat00, FK23}.\\

Let $f \in \mathfrak{g}$ be a locally $\operatorname{ad}$-nilpotent regular element in $U(\mathfrak{g})$. We denote by $U(\mathfrak{g})_{(f)}$ the left ring of fractions of $U(\mathfrak{g})$ with respect to the multiplicative set $\{f^n\ |\  n \in \mathbb{Z}_{\geq 0}\}\subset U(\mathfrak{g})$. In addition, consider a one-parameter family $\{\Theta_f^z\}_{z\in\mathbb{C}}$ of algebra automorphisms of $U(\mathfrak{g})_{(f)}$, defined by

\begin{align*}
  \Theta_f^z(u) = \sum_{k=0}^\infty \binom{z+k-1}{k} f^{-k}\operatorname{ad}(f)^k(u).
\end{align*}

The \emph{twisted localization functor} $D^z_f$, relative to $f$ and $z$, on the category of $\mathfrak{g}$-modules is defined by $ D^z_f(M) = U(\mathfrak{g})_{(f)} \otimes_{U(\mathfrak{g})}M$,
where $M$ is a $\mathfrak{g}$-module. The action of $U(\mathfrak{g})_{(f)}$ on $D^z_f(M)$  is twisted through $\Theta_f^z$, i.e., $u(v) := \Theta_f^z(u)v$
for $u\in U(\mathfrak{g})_{(f)}$ and $v \in D_f^z(M)$. We use the notation $D_f$ instead of $D_f^0$. If $f$ acts injectively on $M$, we can consider $M$ as a submodule of $D_f(M)$ and define the \emph{twisting functor} $T_f$ on the category of $\mathfrak{g}$-modules as $T_f(M) = D_f(M)/M$.\\ 

As a direct consequence of the explicitness of the action of $\mathfrak{gl}_n$ on relation modules, in \cite[$\S$ 5.1]{FMR21} the authors described the localization of several classes of Gelfand-Tsetlin modules with respect to $f=E_{21}$. In what follows, as an application of our results, we construct and provide some structural results for $\sigma$-relation modules localized in the direction of $f=E_{\sigma^{-1}(2),\sigma^{-1}(1)}$.

\begin{theorem}
\label{twistedErs} Let $M$ be a $\sigma$-relation $\mathfrak{gl}_n$-module, and $f:=E_{\sigma^{-1}(2),\sigma^{-1}(1)}$.
\begin{itemize}
\item[(i)] If $f$ acts injectively on $M$, then $T_f(M)$ is a $\sigma$-relation  $\mathfrak{gl}_n$-module.
\item[(ii)] If $f$ acts injectively on $M$, then $D_f^z(M)$ is a $\sigma$-relation $\mathfrak{gl}_n$-module for any $z \in \mathbb{C}$.
\item[(iii)] If $f$ acts bijectively on $M$, then $M\simeq D_f^z(N)$ for some simple  $\sigma$-relation $\mathfrak{gl}_n$-module $N$ with an injective action of $f$ and some $z \in \mathbb{C}\setminus \mathbb{Z}$.
\end{itemize}

\end{theorem}

\begin{proof}
By hypothesis, there is a tableau $T(v)$ such that  $M\simeq (V_{\mathcal{G}}(T(v)))^{\sigma}$, where $\mathcal{G}=G(T(v))$. As the action of $\g$ on $M$ is twisted by $\sigma$, for any $T(u) \in \mathcal{B}_{\mathcal{G}}(T(v))$ we have $f (T(u)) =T(u-\delta^{11})$ (see (\ref{GT-formulas})). The injectivity of $f$ on $M$ implies that $v_{11}-v_{21}, v_{11}-v_{22}\notin \mathbb{Z}_{> 0}$ (see \cite[Lemma 5.1]{FMR21}). Hence, $D_f(M) \simeq (V_{ \widetilde{\mathcal{G}}}(T(v)))^{\sigma}$, where $\widetilde{ \mathcal{G}}$ is the graph obtained from $\mathcal{G}$ by removing the arrows from the second row to the first row (see \cite[Lemma 5.2]{FMR21}).   The first statement follows from the fact that $T_f(M)$ is a quotient of $\sigma$-relation modules. The second statement follows from \cite[Theorem 5.4]{FMR21}; indeed, $D_f^z(M) \simeq (V_{ \widetilde{\mathcal{G}}}(T(v+z\delta^{11})))^{\sigma}$. The third statement follows from \cite[Corolary 5.5 (a)]{FMR21}.  

\end{proof}

\begin{theorem}\label{Th:hw-localization}
 Let $\lambda=(\lambda_1,\dots, \lambda_n)$ be a $\sigma$-relation $\mathfrak{gl}_n$-weight, such that  $s=\sigma^{-1}(1)<\sigma^{-1}(2)=r$, $f=E_{rs}$, and assume that $\lambda_s-\lambda_r \notin \mathbb{Z}_{\geq 0}$. Then 
 
\begin{itemize}
 \item[(i)] $T_f(L(\lambda))$  is a simple $\sigma$-relation $\mathfrak{gl}_n$-module.

 \item[(ii)] $D_f^z(L(\lambda))$ is a $\sigma$-relation $\mathfrak{gl}_n$-module for any $z\in \mathbb{C}$. Moreover, $D_f^z(L(\lambda))$ is simple if and only if $z \notin \mathbb{Z}$ and $z+\lambda_s-\lambda_r \notin \mathbb{Z}$.
 \end{itemize}
\end{theorem}

\begin{proof}
By hypothesis $L(\lambda)\simeq (V_{\mathcal{G}}(T(Y)))^{\sigma}$, where $T(Y)=T_{\sigma}(\bar{\lambda})$ and $\mathcal{G}=G(T(Y))$. As $s=\sigma^{-1}(1)<\sigma^{-1}(2)=r$, we have $A_{(2,1)}(G_\sigma)=1$ (see Theorem \ref{s-act-Gh}) therefore, $f$ acts injectively on $L(\lambda)$ (see \cite[Lemma 5.1]{FMR21}), and does not act surjectively since $\lambda_s-\lambda_r \notin \mathbb{Z}_{\geq 0}$. 
Thus, $T_f(L(\lambda))$ and $D_f^z(L(\lambda))$ are $\sigma$-relation $\mathfrak{gl}_n$-modules by Theorem \ref{twistedErs}. 
Now, $T_f(L(\lambda))$ is simple as $T_f(L(\lambda)) \simeq V_{G_1}(T(Y+\delta^{11}))^{\sigma}$, where $G_1=G(T(Y+\delta^{11})) $. Finally, as $z \notin \mathbb{Z}$ is such that $z+\lambda_s-\lambda_r \notin \mathbb{Z}$, then $D_f^z(L(\lambda)) \simeq V_{G_2}(T(Y+z\delta^{11}))^{\sigma} $, where $G_2=G(T(Y+z\delta^{11})) $, which implies the simplicity of $D_f^z(L(\lambda))$.
\end{proof}

\appendix\section{ }\label{appendix}

In this appendix, we provide the proofs of Lemma \ref{Lemma: sk and tableaux} and Lemma \ref{PropTech03}.

\begin{lemma}\label{cor: sum of arrows} Set  $G\in \tilde{\Sigma}$, and $1\leq r< j< i$.
\begin{itemize}
    \item[(i)] $A_{(j,r)}(G) = 1$, and $A_{(i,r)}(G) = -1$ imply that $A_{(i,j)}(G) = -1$.
    \item[(ii)]  $A_{(j,r)}(G) = -1$, and $A_{(i,r)}(G) = 1$ imply that $A_{(i,j)}(G) = 1$.
    \item[(iii)]  If  $A_{(i,r)}(G) = 1$, and $A_{(i,j)}(G) = -1$, then $A_{(j,r)}(G) = 1$.
     \item[(iv)]  If  $A_{(i,r)}(G) =- 1$, and $A_{(i,j)}(G) = 1$, then $A_{(j,r)}(G) = -1$.
\end{itemize}
\end{lemma}

\begin{proof} We provide a proof only for item (i) since (ii), (iii), and (iv) are analogous.
The conditions imposed on $G$ imply that $\alpha_{r,j}, \alpha_{i,r} \in \mathcal{Q}_G$. Since $\mathcal{Q}_G$ is a closed partition, $\alpha_{i,j}=\alpha_{r,j}+\alpha_{i,r}\in \mathcal{Q}_G$, and so $A_{(i,j)}(G) = -1$.
\end{proof}

\begin{lemma}\label{lem: Mk, mk different}
Let $G=G_{\sigma}$ for some $\sigma\in S_n$, and $k\leq n-1$.
\begin{itemize}
\item[(i)] If $A_{(k+1,k)}(G)=1$, then:
    \begin{itemize}
\item[(a)] $A_{(k+1,\ell)}(G)\neq A_{(k,\ell)}(G)$ $\iff$  $1\leq \ell<k$,  and  $\sigma^{-1}(k)<\sigma^{-1}(\ell)<\sigma^{-1}(k+1)$.
\item[(b)] $A_{(\ell,k)}(G)\neq A_{(\ell, k+1)}(G)$ $\iff$ $k+1 <\ell\leq n$, and  $\sigma^{-1}(k)<\sigma^{-1}(\ell)<\sigma^{-1}(k+1)$. 
    \end{itemize}
\item[(ii)] If $A_{(k+1,k)}(G)=-1$, then:
\begin{itemize}
\item[(a)] $A_{(k+1,\ell)}(G)\neq A_{(k,\ell)}(G)$ $\iff$ $1\leq \ell<k$,  and  $\sigma^{-1}(k+1)<\sigma^{-1}(\ell)<\sigma^{-1}(k)$.
\item[(b)] $A_{(\ell,k)}(G)\neq A_{(\ell, k+1)}(G)$ $\iff$ $\ell\geq k+2$, and  $\sigma^{-1}(k+1)<\sigma^{-1}(\ell)<\sigma^{-1}(k)$. 
\end{itemize}
\end{itemize}
\end{lemma}

\begin{proof} We prove (i)(a) and (ii)(a), since the proofs of the other cases are analogous.

\textbf{Case 1: $A_{(k+1,k)}(G)=1$.}\\ 
 If $\sigma^{-1}(k)<\sigma^{-1}(\ell)<\sigma^{-1}(k+1)$, by Theorem \ref{s-act-Gh} we have $A_{(k, \ell)}(G)=-1$ and $A_{(k+1,\ell)}(G)=1$. Conversely, the case $A_{(k, \ell)}(G)=1$ and $A_{(k+1,\ell)}(G)=-1$, together with Lemma \ref{cor: sum of arrows}(i), imply $A_{(k+1,k)}(G)=-1$, which is a contradiction. Finally, note that $\sigma^{-1}(\ell)<\sigma^{-1}(k)$ implies $A_{(k, \ell)}(G)=1$, and $\sigma^{-1}(\ell)>\sigma^{-1}(k+1)$, implies that  
 $A_{(k+1,\ell)}(G)=-1$. Therefore, $\sigma^{-1}(k)<\sigma^{-1}(\ell)<\sigma^{-1}(k+1)$.

\textbf{Case 2: $A_{(k+1,k)}(G)=-1$.} \\ 
 If $\sigma^{-1}(k+1)<\sigma^{-1}(\ell)<\sigma^{-1}(k)$, by Theorem \ref{s-act-Gh} we have $A_{(k+1, \ell)}(G)=-1$, and $A_{(k,\ell)}(G)=1$. Conversely, the case $A_{(k+1, \ell)}(G)=1$, and $A_{(k,\ell)}(G)=-1$, together with Lemma \ref{cor: sum of arrows} (ii) implies $A_{(k+1,k)}(G)=1$, which is a contradiction. Now suppose that $\sigma^{-1}(\ell)<\sigma^{-1}(k+1)$, then $A_{(k+1, \ell)}(G)=1$, by Theorem \ref{s-act-Gh}, but this is a contradiction, hence $\sigma^{-1}(k+1)<\sigma^{-1}(\ell)$.  Finally, note that $\sigma^{-1}(\ell)<\sigma^{-1}(k+1)$ implies $A_{(k+1, \ell)}(G)=1$, and $\sigma^{-1}(\ell)>\sigma^{-1}(k)$, implies that  
 $A_{(k,\ell)}(G)=-1$. Therefore, $\sigma^{-1}(k+1)<\sigma^{-1}(\ell)<\sigma^{-1}(k)$.
\end{proof}
 \begin{definition}
 Given $G\in \tilde{\Sigma}$ and $k\in \{1, \dots, n-1\}$ define:
 \begin{align*}
  M_{k}(G)&:=A_{(k+1,k)}(G) \cdot \#\{\ell\ |\ k+2\leq \ell\leq  n,\text{ and } A_{(\ell,k)}(G)\neq A_{(\ell, k+1)}(G)\},\\
 m_{k}(G)&:=A_{(k+1,k)}(G) \cdot \#\{\ell\ |\ 1\leq \ell\leq k-1,\text{ and } A_{(k+1,\ell)}(G)\neq A_{(k,\ell)}(G)\}.
 \end{align*}
 \end{definition}

\begin{lemma}\label{Lemma: dual}
    Set $X\in\mathbb{C}^n$, $\sigma\in S_n$,  $T(Y):=T_{\sigma}(X)$, and $1\leq k\leq n-1$.
\begin{itemize}
    \item [(i)] $M_{k}(G_\sigma) =\displaystyle \frac{1}{2} \sum_{i=k+2}^n(A_{(i,k)}(G_{\sigma}) - A_{(i,k+1)}(G_{\sigma}))$;
    \item [(ii)] $m_{k}(G_\sigma) = \displaystyle \frac{1}{2} \sum_{i=1}^{k-1}(A_{(k+1,i)}(G_{\sigma}) - A_{(k,i)}(G_{\sigma}) )$;
    \item [(iii)] $M_{k}(G_\sigma) +m_{k}(G_\sigma)+ A_{(k+1,k)}(G_\sigma)= \sigma^{-1}(k+1)-\sigma^{-1}(k)$.
    \item[(iv)] $y_{k+1,k+1} - y_{k+1,k} = y_{n,k+1} - y_{n,k} + M_{k}(G_\sigma).$
\end{itemize}
    
\end{lemma}

\begin{proof}
We provide a proof of the lemma under the assumption that $A_{(k+1,k)}(G_\sigma)=1$, since the case $A_{(k+1,k)}(G_\sigma)=-1$ is analogous. In this case, Theorem \ref{s-act-Gh} implies that  $\sigma^{-1}(k)<\sigma^{-1}(k+1)$, and for $\ell \notin \{ k, k+1 \}$, we have three possibilities: (a) $\sigma^{-1}(\ell)<\sigma^{-1}(k)<\sigma^{-1}(k+1)$; (b) $\sigma^{-1}(k)<\sigma^{-1}(\ell)<\sigma^{-1}(k+1)$; (c) $\sigma^{-1}(k)<\sigma^{-1}(k+1)<\sigma^{-1}(\ell)$.
To prove item (i) we consider $\ell \geq k+2$, and use Theorem \ref{s-act-Gh} in a case by case consideration to obtain that $A_{(\ell,k)}(G_{\sigma}) =A_{(\ell,k+1)}(G_{\sigma})=-1$, in case (a), $A_{(\ell,k)}(G_{\sigma}) =1$  and $A_{(\ell,k+1)}(G_{\sigma})=-1$, in case (b), and $A_{(\ell,k)}(G_{\sigma}) =A_{(\ell,k+1)}(G_{\sigma})=1$ in case of (c). Then for any $\ell\geq k+2$, we have $A_{(\ell,k)}(G_{\sigma}) - A_{(\ell,k+1)}(G_{\sigma})=2$ whenever $A_{(\ell,k)}(G_{\sigma})\neq A_{(\ell,k+1)}(G_{\sigma})$. In the case $A_{(k+1,k)}(G_\sigma)=-1$ we obtain $A_{(\ell,k)}(G_{\sigma}) - A_{(\ell,k+1)}(G_{\sigma})=-2$ whenever $A_{(\ell,k)}(G_{\sigma})\neq A_{(\ell,k+1)}(G_{\sigma})$, so the formula follows after multiplying by $A_{(k+1,k)}(G_\sigma)$. The prove of (ii) uses analogous arguments.

For (iii), by Lemma \ref{lem: Mk, mk different} (i), we get $M_{k}(G_\sigma)=\#\{\ell\geq k+2\ |\ \sigma^{-1}(k)<\sigma^{-1}(\ell)<\sigma^{-1}(k+1)\}$, and $m_{k}(G_\sigma)=\#\{\ell\leq k-1\ |\ \sigma^{-1}(k)<\sigma^{-1}(\ell)<\sigma^{-1}(k+1)\}$.

To prove (iv) we use Lemma \ref{y_ij}, and item (i) to obtain
\begin{align*}
     y_{k+1,k+1} - y_{k+1,k} &=\left(y_{n,k+1}  + \frac{1}{2}\sum_{i=k+2}^n (1-A_{(i,k+1)}(G_{\sigma
}))\right)-\left(y_{n,k}  + \frac{1}{2}\sum_{i=k+2}^n(1- A_{(i,k)}(G_{\sigma
}) )\right) \\
     &= y_{n,k+1} - y_{n,k} + \frac{1}{2}\sum_{i=k+2}^n(A_{(i,k)}(G_{\sigma}) - A_{(i,k+1)}(G_{\sigma}))\\
     &=y_{n,k+1} - y_{n,k} + M_{k}(G_\sigma).
\end{align*}
\end{proof}

\subsection{Proof of Lemma \ref{Lemma: sk and tableaux}} Recall that for any tableau $T(R)$, $\Sigma_{t}(R)$ denotes the sum of the entries in row $t$ of $T(R)$.
\begin{lemma*}Let $T(R)=T_{\tau}(X)$ and $T(W)=T_{s_k\circ\tau}(X)$ respectively, for some $\tau\in S_n$, and a simple transposition $s_k$.

\begin{itemize}
\item[(i)] $\Sigma_{i}(W)=\Sigma_{i}(R)$ for $i\neq k$.

\item[(ii)] $\Sigma_{k}(W)=\Sigma_{k}(R)+ r_{n,k+1}- r_{n,k} + \tau^{-1}(k+1)-\tau^{-1}(k)$.
\end{itemize}

\end{lemma*}
\begin{proof}
Item (i) follows from Corollary \ref{co: s-act-Gh}.  Indeed, $w_{ij}=r_{ij}$ whenever $j\notin\{k,k+1\}$, $i\neq k$; and $w_{i,k+1}=r_{ik}$,  $w_{i,k}=r_{i,k+1}$, whenever $i\geq k+1$. For item (ii), we compute explicitly $w_{k,j}$ for $j\leq k$.

\begin{itemize}
\item[(a)] Suppose first $1\leq j\leq k-1$. We use Corollary \ref{co: s-act-Gh} and Formula (\ref{y_ij explicit}) to get 
\begin{align*}
    w_{k,j} =& w_{n,j} + \frac{1}{2} \sum_{\ell = k+1}^{n}(1-A_{(\ell ,j)}(G_{s_k\circ\tau}))\\
    =& r_{n,j} + \frac{1}{2} \sum_{\ell = k+2}^{n}(1-A_{(\ell ,j)}(G_{\tau})) + \frac{1}{2}(1-A_{(k,j)}(G_{\tau}))\\
    =& r_{n,j} + \frac{1}{2} \sum_{\ell = k+1}^{n}(1-A_{(\ell ,j)}(G_{\tau})) + \frac{1}{2}(1-A_{(k,j)}(G_{\tau}))-\frac{1}{2}(1-A_{(k+1,j)}(G_{\tau}))\\
    =& r_{k,j} + \frac{1}{2}\left(A_{(k+1,j)}(G_{\tau}) - A_{(k,j)}(G_{\tau})\right).
\end{align*}
\item[(b)] To compute $w_{kk}$, we use Corollary \ref{co: s-act-Gh}, Lemma \ref{y_ij}, and Lemma \ref{Lemma: dual} (iv)  to get 
$$w_{kk}=r_{k,k}+r_{n,k+1} - r_{n,k} + M_{k}(G_\tau)+A_{(k+1,k)}(G_{\tau}).$$

\end{itemize}
 For item (ii), we use (a), (b), and Lemma \ref{Lemma: dual} (ii), (iii) to obtain that:
\begin{align*}
    \Sigma_{k}(W) &= \sum_{j=1}^{k} w_{kj}=\sum_{j=1}^{k-1} w_{kj}+w_{kk}\\
    & =\sum_{j=1}^{k-1} \left(r_{kj}+  \frac{1}{2}(A_{(k+1,j)}(G_{\tau}) - A_{(kj)}(G_{\tau})) \right)+w_{kk}\\
    & =\sum_{j=1}^{k-1} r_{kj}+m_{k}(G_\tau)+r_{k,k}+r_{n,k+1} - r_{n,k} + M_{k}(G_\tau)+A_{(k+1,j)}(G_{\tau})\\
     & =\Sigma_{k}(R)+r_{n,k+1} - r_{n,k}+m_{k}(G_\tau) + M_{k}(G_\tau)+A_{(k+1,j)}(G_{\tau})\\
     &=\Sigma_{k}(R)+ r_{n,k+1}- r_{n,k} + \tau^{-1}(k+1)-\tau^{-1}(k).
\end{align*}
The last equality follows from Lemma \ref{Lemma: dual} (iii).
\end{proof}

\subsection{Proof of Lemma \ref{PropTech03}}
\begin{lemma*} Set $\sigma\in S_n$, $X\in\mathbb{C}^n$, $G:=G_{\sigma}$, and $T(L):=T_{\sigma}(X)$. Given $r<s$, and $\{i_r,i_{r+1},\ldots,i_{s-1}\}$ with $1\leq i_t\leq t$ and $r\leq t <s$, we have:
\begin{itemize}
\item[(i)] If $A_{(s,r)}(G)=1$, then  $T(L+\delta^{r,i_r}+\ldots+\delta^{s-1,i_{s-1}})$ does not  satisfy  $G$.
\item[(ii)] If $A_{(s,r)}(G)=-1$, then $T(L-\delta^{r,i_r}-\ldots-\delta^{s-1,i_{s-1}})$ does not  satisfy $G$.
\end{itemize}

\end{lemma*}
\begin{proof}
The proof of (i) and (ii) is analogous, so we provide a detailed proof only for (i). The proof will be divided in two steps. Suppose $A_{(s,r)}(G)=1$.
\begin{enumerate}[Step $1.$]
\item $T(L+\delta^{r,i}+\ldots+\delta^{s-1,i})$ does not  satisfied $G$ for any $i\leq r$:\\
A direct verification shows that $A_{(r,i)}(G)=-1$, or $A_{(s,i)}(G)=1$ implies that $T(L+\delta^{r,i}+\ldots+\delta^{s-1,i})$ does not satisfy $G$. Let us assume that $A_{(r,i)}(G)=1$, and $A_{(s,i)}(G)=-1$. In this case, by Lemma \ref{cor: sum of arrows}(i) we should have $A_{(s,r)}(G)=-1$, which contradicts the first part of the hypothesis.
\item $T(L+\delta^{r,i_r}+\ldots+\delta^{s-1,i_{s-1}})$ does not  satisfied  $G$ for any $\{i_r,i_{r+1},\ldots,i_{s-1}\}$:\\
We proceed by induction on $t=s-r$. The case $t=1$ follows from Step $1$.
Suppose now that $t>1$, and the statement of the lemma is true for $k<t$. If there is $\{i_r,i_{r+1},\ldots,i_{s-1}\}$ with $1\leq i_a\leq a$ such that $T(L+\delta^{r,i_r}+\ldots+\delta^{s-1,i_{s-1}})$ satisfies  $G$, then 
by Step $1$, there exists $k$ such that $i_{k}\neq i_{s-1}$. Set $k_{0}$ to be the maximum of $\{k\ |\  i_{k}\neq i_{s-1}\}$.
Under this conditions \begin{itemize}
\item[(a)] $T(L+\delta^{k_0+1,i_{k_{0}+1}}+\ldots+\delta^{s-1,i_{s-1}})$ satisfies $G$, and by Step $1$, we have that $A_{(s,k_0+1)}(G)=-1$. 
\item[(b)] $T(L+\delta^{r,i_r}+\ldots+\delta^{k_0,i_{k_0}})$ satisfies  $G$, which by induction hypothesis implies that  $A_{(k_0+1,r)}(G)=-1$.
\end{itemize}
However, by Lemma \ref{cor: sum of arrows} (iii), $A_{(s,r)}(G)=1$, and $A_{(s,k_0+1)}(G)=-1$ necessarily implies $A_{(k_0+1,r)}(G)=1$, which contradicts (b).
\end{enumerate}
\end{proof}

\bibliographystyle{alpha}

\end{document}